\documentclass{amsart}

\usepackage{amsmath}
\usepackage{amsxtra}
\usepackage{amscd}
\usepackage{amsthm}
\usepackage{amsfonts}
\usepackage{amssymb}
\usepackage{color}
\usepackage{eucal}
\usepackage[all]{xy}
\usepackage{graphicx}
\usepackage{mathabx}
\usepackage{hyperref}
\usepackage{epsfig}

\newtheorem{theorem}{Theorem}[section]
\newtheorem*{theorem*}{Theorem}
\newtheorem{corollary}[theorem]{Corollary}

\newtheorem{lemma}[theorem]{Lemma}
\newtheorem{proposition}[theorem]{Proposition}
\newtheorem*{proposition*}{Proposition}

\newtheorem*{conjecture*}{Conjecture}

\newtheorem{thm-dfn}[theorem]{Theorem-Definition}
\newtheorem{clm-dfn}[theorem]{Claim-Definition}
\newtheorem{claim}[theorem]{Claim}

\theoremstyle{definition}

\newtheorem{remark}[theorem]{Remark}

\numberwithin{equation}{section}

\newcommand{\quash}[1]{}  %Anything in \quash is ignored

\newcommand{\bs}{\backslash}

\newcommand{\frakm}{{\mathfrak m}}
\newcommand{\frakn}{{\mathfrak n}}

\newcommand{\bbA}{{\mathbb A}}

\newcommand{\bbC}{{\mathbb C}}

\newcommand{\bbG}{{\mathbb G}}

\newcommand{\bbL}{{\mathbb L}}

\newcommand{\bbP}{{\mathbb P}}
\newcommand{\bbQ}{{\mathbb Q}}
\newcommand{\bbR}{{\mathbb R}}

\newcommand{\bbU}{{\mathbb U}}

\newcommand{\bbZ}{{\mathbb Z}}

\newcommand{\calA}{{\mathcal A}}

\newcommand{\calD}{{\mathcal D}}

\newcommand{\calH}{{\mathcal H}}
\newcommand{\calI}{{\mathcal I}}

\newcommand{\calL}{{\mathcal L}}
\newcommand{\calM}{{\mathcal M}}

\newcommand{\calO}{{\mathcal O}}
\newcommand{\calP}{{\mathcal P}}

\newcommand{\Hom}{\textnormal{Hom}}
\newcommand{\End}{\textnormal{End}}

\newcommand{\Fun}{\textnormal{Fun}}

\newcommand{\Vect}{\textnormal{Vect}}
\newcommand{\op}{\textnormal{op}}

\newcommand{\colim}{{\textnormal{colim}}}
\newcommand{\Ker}{{\textnormal{Ker}}}
\newcommand{\Image}{{\textnormal{Im}}}

\newcommand{\Alg}{{\textnormal{Alg}}}
\newcommand{\Aff}{{\textnormal{Aff}}}

\newcommand{\pc}{{\textnormal{pc}}}
\newcommand{\fl}{{\textnormal{fl}}}
\newcommand{\sm}{{\textnormal{sm}}}
\newcommand{\alg}{{\textnormal{alg}}}

\newcommand{\PD}{{\textnormal{PD}}}

\begin{document}

\title[A spherical $p$-adic Paley-Wiener theorem and Bernstein morphisms]{A Paley-Wiener theorem for spherical $p$-adic spaces and Bernstein morphisms}
\author{Alexander Yom Din}

\maketitle

\begin{abstract}
	Let $G$ be (the rational points of) a connected reductive group over a local non-archimedean field $F$. In this article we formulate and prove a property of an $F$-spherical homogeneous $G$-space (which in addition satisfies the finite multiplicity property, which is expected to hold for all $F$-spherical homogeneous $G$-spaces) which we call the Paley-Wiener property. This is much more elementary, but also contains much less information, than the recent relevant work of Delorme, Harinck and Sakellaridis (however, it holds for a wider class of spaces). The property results from a parallel categorical property. We also discuss how to define Bernstein morphisms via this approach.
\end{abstract}

\tableofcontents

%%%%%%%%%%%%%%%%%%%%%%%%%%%%%%%%%%%%
%%%%%%%%%%%%%%%%%%%%%%%%%%%%%%%%%%%%

\section{Introduction}

\subsection{Overview}

\subsubsection{}

The prototype of a Paley-Wiener problem is the following. Consider the functions $$ \phi_z \in C^{\infty}(\bbR), \quad \phi_z (x) := e^{2 \pi i xz}$$ for $z \in \bbC$. Given a smooth function with compact support $f \in C^{\infty}_c (\bbR)$, one can define numbers (the Fourier transform) $$c_z := \int_{\bbR} f(x) \cdot \phi_z (x) \cdot dx$$ for $z \in \bbC$. The Paley-Wiener problem seeks to characterize collections $(c_z)_{z \in \bbC}$ for which there exists a smooth function with compact support $f \in C^{\infty}_c (\bbR)$ with the relation above.

\subsubsection{}

Let us now switch to a non-archimedean local field $F$ (such as the field $\bbQ_p$ of $p$-adic numbers), rather than the field $\bbR$ of real numbers, and to the multiplicative group $F^{\times}$ rather than the additive group $F$. Functions we consider are smooth $\bbC$-valued functions, where ``smooth" now means ``locally constant". Analogously to the functions $(\phi_z)_{z \in \bbC}$ above, one considers the smooth characters $F^{\times} \to \bbC^{\times}$. Given a smooth character $\chi_1$, we obtain a whole family of such, considering for $z \in \bbC^{\times}$ the character $\chi_z (x) :=  \chi_1 (x) \cdot z^{v(x)}$ (where $v : F^{\times} \to \bbZ$ stands for the valuation). Therefore, given a smooth function with compact support $f \in C^{\infty}_c (F^{\times})$, we define again numbers $$ c_{\chi_z} := \int_{F^{\times}} f(x) \cdot \chi_z (x) \cdot dx.$$ The analog of the Paley-Wiener problem will ask to characterize collections $(c_{\chi_z})$ originating from $f$ as above. The main condition for that turns out to be that $z \mapsto c_{\chi_z}$ should be an algebraic ($=$regular) function on $\bbC^{\times}$.

\subsubsection{}

We now generalize and consider, instead of $F^{\times}$, the group $G$ of $F$-points of a connected reductive group over $F$, such as $GL_n (F)$. The analogue of the above collcetion of characters is the space $\calA (G)$ of matrix coefficients of smooth (complex) $G$-representations of finite length (roughly speaking, those are functions all of whose translates span a representation close to being irreducible, so again ``eigen-functions" of some sort). One again wants to ask, given a functional on the space $\calA (G)$, what conditions it should satisfy in order to be equal to integration against a smooth function with compact support on $G$. Such a Paley-Wiener theorem was established by Bernstein (\cite[\S 5.2]{BeRu}) and Heiermann (\cite{He}). The main condition is, similarly to above, that by evaluating the functional on some specific families of functions in $\calA (G)$, parametrized by complex tori - families of matrix coefficients of parabolic induction - one should obtain algebraic functions on the tori.

\subsubsection{}

Further generalizing, one considers an $F$-spherical right homogeneous $G$-space $H \backslash G$ (here $H \subset G$ is a closed subgroup). Spherical homogeneous spaces are an important class of homogeneous spaces for which, it is understood today, results and theories similar to ones for the group $G$ itself should/can be obtained (this might be called the ``relative point of view" in this context). One again can consider the space of ``eigen-functions" $\calA (H \backslash G)$, and ask for conditions on a functional on that space to be equal to integration again a smooth function with compact support. The paper \cite{DeHaSa} has dealt with the Paley-Wiener problem in this context. Important ingredients are boundary degenerations and Bernstein morphisms, which morally provide families of functions on which to test our functional.

\subsubsection{}

In this article, we would like to state a Paley-Wiener result for an $F$-spherical homogeneous $G$-space which is more elementary, and less precise, than the results of \cite{DeHaSa} (but is valid for a larger class of spaces). The main idea is, roughly speaking, that instead of studying boundary degenerations and Bernstein morphisms in order to produce specific families of functions in $\calA (H \backslash G)$ to test on, we will allow arbitrary families (hence, a bit similarly to defining a manifold structure on a set abstractly, without constructing specific coordinate charts). However, unfolding our proof, one sees that implicitly we still use specific families, namely the parabolic induction families for $G$ itself, as above. So, intuitively speaking, the idea is that the sphericity of $H \backslash G$ provides that the situation for it is ``covered well" by the situation for $G$ itself, even if one does not know finer information.

\subsubsection{}

For our result, we have to assume that our $F$-spherical homogeneous $G$-space satisfies the finite multiplicity property (which is expected to hold for all $F$-spherical $G$-spaces), and use in addition the ``local finite-generation" property of such spaces, given by \cite[Theorem I]{AiGoSa}.

\subsubsection{}

We will also describe how to construct Bernstein morphisms using the technique of this article (Bernstein morphisms were defined in works \cite{BeKa} and (in greater generality) \cite{SaVe}, being called in the former ``Bernstein maps" and in the latter ``asymptotics maps", as the name ``Bernstein maps" is reserved there for the $L^2$-version). We provide a ``preparation", so to speak, by defining what constitutes a suitable degeneration along which we then define the Bernstein morphism. The detailed discussion of the relation to a parametrization of degenerations (say in terms of subsets of the set of spherical roots) we hope to provide elsewhere.

\subsection{Contents}

In this description of the contents of the article, we allow ourselves to be loose regarding the actual technicalities.

\subsubsection{}

In \S\ref{section abstract PW}, we will formulate a property of a linearly topologized vector space $\calA$, which we call the Paley-Wiener property. Roughly, we first define what is an algebraic mapping from an affine variety $Z$ to $\calA$, and then say that a functional on $\calA$ is algebraic if, given any algebraic mapping to $\calA$ from any smooth $Z$, composing it with the functional one obtains an algebraic function on $Z$. Then we say that $\calA$ is Paley-Wiener if every algebraic functional on it is continuous.

\subsubsection{}

In \S\ref{section sph PW}, we will formulate the main theorem, which states that if $H \backslash G$ is an $F$-spherical $G$-space and it has the finite multiplicity property (which is expected to hold for all $F$-spherical $G$-spaces), then $\calA (H \backslash G)$ is Paley-Wiener (in the sense of \S\ref{section abstract PW}). Unfolded, this gives a criterion for a functional on $\calA (H \backslash G)$ to be equal to integration against a smooth distribution with compact support on $H \backslash G$. As we mentioned above, the criterion is abstract, in a sense - one considers arbitrary algebraic families of functions in $\calA(H \backslash G)$.

\subsubsection{}

In \S\ref{sec categorical}, we consider a categorical conception related to the Paley-Wiener properties above. Namely, given an abelian category $\calM$ and functors $C,D$ from $\calM$ to the category of vector spaces, we define the property of a morphism $$D|_{\calM^{\fl}} \to C|_{\calM^{\fl}}$$ to be algebraic (here $\calM^{\fl} \subset \calM$ stands for the full subcategory of objects of finite length), which morally means that it interacts well with algebraic families of objects of finite length. We are then interested in conditions for the image of the restriction map $$ \Hom (D , C) \to \Hom (D|_{\calM^{\fl}} , C|_{\calM^{\fl}})$$ to coincide with the subspace of algebraic morphisms.

\subsubsection{}

In \S\ref{sec proof of main thm} we prove the main theorem (that of \S\ref{section sph PW}, regarding the Paley-Wiener property of an $F$-spherical homogeneous $G$-space), recasting it into the formalism of \S\ref{sec categorical}, and using theory based on Bernstein's decomposition of the category of smooth $G$-modules with respect to the center.

\subsubsection{}

In \S\ref{sec Bernstein morphisms} we describe  a setup for degenerations of $F$-spherical subgroups and the construction of Bernstein morphisms along such degenerations, using the technique of this article.

\subsubsection{}

In \S\ref{sec remaining proofs} we provide proofs for statements which were left without proof earlier in the article.

\subsubsection{}

In \S\ref{sec algebraic lemma} we give technical algebraic propositions, on reconstruction of a global section of a coherent sheaf on an affine variety from sections on closed subvarieties, which are used in the proof of the main proposition of \S\ref{sec categorical}.

\subsection{A few general notations and conventions}

\subsubsection{}

All objects (vector spaces, algebras, function spaces, and so on) will be over $\bbC$.

\subsubsection{}

$\bbG$ will denote a connected reductive group over a local non-archimedean field $F$, and we denote $G := \bbG (F)$. $\calH_G$ will denote the Hecke algebra of smooth distributions with compact support on $G$. For an $\ell$-space $X$ we will denote by $C^{\infty} (X)$ the space of locally constant functions on $X$.

\subsubsection{}

$\Alg$ will denote the category of finitely generated $\bbC$-algebras. $\Aff$ will denote the category of affine schemes of finite type over $\bbC$. For $Z \in \Aff$, we will denote by $\calO_Z \in \Alg$ the corresponding algebra.

\subsubsection{}

Given $A \in \Alg$ and an ideal $I \subset A$, we will say that $I$ is cofinite if the dimension of the $\bbC$-vector space $A/I$ is finite.

\subsubsection{}

For a unital $\bbC$-algebra $R$, we will denote by $\calM (R)$ the $\bbC$-linear abelian category of left $R$-modules. We will denote by $\calM (G)$ the $\bbC$-linear abelian category of smooth left $G$-modules. When we say ``$G$-module", we mean a not necessarily smooth one. Given a (say, left) $G$-module $M$, we will denote by ${}^{\sm} M$ the submodule of smooth vectors, which forms a smooth $G$-module.

\subsection{Acknowledgments}

This material is based upon work supported by the Institute for Advanced Study School of Mathematics. I would like to thank H\'el\`ene Esnault for a useful discussion regarding the contents of \S\ref{sec algebraic lemma}.

%%%%%%%%%%%%%%%%%%%%%%%%%%%%%%%%%%%%
%%%%%%%%%%%%%%%%%%%%%%%%%%%%%%%%%%%%

\section{The abstract Paley-Wiener property for linearly topologized vector spaces}\label{section abstract PW}

In this section we formulate an abstraction of the Paley-Wiener property that we wish to study.

\subsection{Linearly topologized vector spaces}

\subsubsection{}

An \textbf{LTVS} (here ``LTVS" stands for "Hausdorff linearly topologized vector space") is a Hausdorff topological $\bbC$-vector space - where $\bbC$ is considered with the discrete topology - for which the open vector subspaces form a basis of neighbourhoods of $0$ (and morphisms of LTVS are continuous linear maps). A finite dimensional vector space admits a unique LTVS topology (the discrete one), and hence the category of finite dimensional LTVS is the same as that of finite dimensional vector spaces.

\subsubsection{}

We will say that an LTVS is \textbf{pre-linearly compact} if all its open subspaces are of finite codimension. This is equivalent to the completion of it being linearly compact, following a common terminology. More concretely, this means that it can be realized as a (not necessarily closed) subspace of a product of finite-dimensional LTVS.

\subsection{Punctual distributions}

Throughout this subsection, we fix $Z \in \Aff$.

\subsubsection{}

A functional $\lambda : \calO_Z \to \bbC$ will be called a \textbf{punctual distribution} if it factors via $\calO_Z \to \calO_Z / I$ for some cofinite ideal $I \subset \calO_Z$. We denote by $$ \PD (Z) \subset \calO_Z^*$$ the subspace of punctual distributions.

\subsubsection{}

We have a tautological injective map $$ Z(\bbC) \hookrightarrow \PD(Z)$$ (because $Z(\bbC)$ can be interpreted as the space of algebra homomorphisms from $\calO_Z$ to $\bbC$).

\subsubsection{}\label{sssec algebraic PDstar}

We have a tautological injective linear map $$ \calO_Z \hookrightarrow \PD(Z)^*$$ (injectivity follows from Krull's intersection theorem). We say that a functional in $\PD(Z)^*$ is \textbf{algebraic} if it lies in the image of this map.

\subsubsection{}\label{sssec embed O in func}

Therefore, we have $$ \calO_Z \hookrightarrow \PD(Z)^* \to \{ \textnormal{functions on } Z(\bbC) \},$$ but the composition is not injective unless $Z$ is reduced.

\subsubsection{}

Suppose that $Z$ is smooth. Then given an algebraic differential operator $D$ on $Z$ and $z \in Z(\bbC)$, we have a punctual distribution $f \mapsto (Df)(z)$. Moreover, any punctual distribution is a finite sum of such.

\subsection{Algebraic mappings}

Throughout this subsection, we fix $Z \in \Aff$ and an LTVS $E$.

\subsubsection{}

We define an \textbf{algebraic mapping} from $Z$ to $E$ to be a linear map $$ \Phi : \PD(Z) \to E$$ for which the following conditions are satisfied:

\begin{enumerate}
    \item The image of $\Phi$ is pre-linearly compact.
    \item For every continuous functional $\alpha : E \to \bbC$, the composition $$\alpha \circ \Phi \in \PD(Z)^*$$ is algebraic (in the sense of \S\ref{sssec algebraic PDstar}).
\end{enumerate}

\subsubsection{}

If $E$ is finite dimensional, the space of algebraic mappings from $Z$ to $E$ can be identified with $\calO_Z \otimes E$.

\subsubsection{}\label{sssec alg map for reduced}

If $Z$ is reduced, the space of algebraic mappings from $Z$ to $E$ can be identified with the space of functions $\Phi : Z(\bbC) \to E$ for which the following conditions are satisifed:

\begin{enumerate}
    \item The span of the image of $\Phi$ is pre-linearly compact.
    \item For every continuous functional $\alpha : E \to \bbC$, the composition $$\alpha \circ \Phi : Z(\bbC) \to \bbC$$ is an algebraic function (i.e., belongs to $\calO_Z$ via the embedding of \S\ref{sssec embed O in func}).
    \item For every $\lambda \in \PD(Z)$, there exists $\phi_{\lambda} \in E$ such that for every continuous functional $\alpha : E \to \bbC$, one has $\alpha (\phi_{\lambda}) = \lambda(\alpha \circ \Phi)$, where here $\alpha \circ \Phi$ is considered as an element of $\calO_Z$, thanks to (2).
\end{enumerate}

\subsection{Algebraic functionals and Paley-Wiener property}

\subsubsection{}\label{sssec algebraic functional}

Let $E$ be a LTVS. A not necessarily continuous functional $\ell : E \to \bbC$ is said to be \textbf{algebraic} if for every smooth $Z \in \Aff$ and every algebraic mapping $\Phi : \PD(Z) \to E$, the composition $\ell \circ \Phi \in \PD (Z)^*$ is algebraic.

\medskip

Clearly continuous functionals are algebraic.

\begin{remark}
	One could also define a functional to be algebraic similarly, but running over all $Z \in \Aff$, not only smooth ones. We did not figure out what would be a better definition (or under what conditions those are the same). In our case of interest we will have the Paley-Wiener property (see next), which in particular gives that functionals algebraic in the chosen sense are also algebraic in the more restrictive sense.
\end{remark}

\subsubsection{}

Let $E$ be a LTVS. We say that $E$ is \textbf{Paley-Wiener} if every algebraic functional on $E$ is continuous.

\subsubsection{}\label{sssec def of Paley-Wiener}

Given an LTVS $F$ and a not necessarily closed subspace $E \subset F$, we will say that the inclusion $E \subset F$ is \textbf{Paley-Wiener} if $E$ is dense in $F$ and $E$ is Paley-Wiener. In other words, the inclusion $E \subset F$ is Paley-Wiener if algebraic functionals on $E$ are the same as continuous functionals on $F$.

\medskip

For intuition, the situation we think of is when $F$ is the space of ``all functions on a space" and $E$ is the subspace of ``eigen-functions". Then continuous functionals on $F$ are some kind of distributions, and the Paley-Wiener property says that functionals on $E$ which are algebraic, are equal to integration against a distribution.

%%%%%%%%%%%%%%%%%%%%%%%%%%%%%%%%%%%%
%%%%%%%%%%%%%%%%%%%%%%%%%%%%%%%%%%%%

\section{The Paley-Wiener property of a spherical $p$-adic space}\label{section sph PW}

In this section we will state the main theorem of the article, Theorem \ref{thm main theorem}. Throughout this section, we fix a homogeneous right $G$-space $X$.

\subsection{Function spaces}\label{ssec function spaces}

\subsubsection{}

We consider the left smooth $G$-module ${}^{\sm} C^{\infty}(X)$. We equip it with an LTVS structure as follows. For an open compact subgroup $K \subset G$, give ${}^K C^{\infty}(X) \cong C^{\infty}(X / K)$ the linearly compact LTVS topology making all point evaluations continuous. Then give ${}^{\sm} C^{\infty} (X)$ the colimit linear topology $$ {}^{\sm} C^{\infty} (X) = \underset{K \subset G}{\colim} \ {}^K C^{\infty} (X).$$ This linear topology is Hausdorff, since the subspaces of functions which vanish on a fixed compact are open, and their intersection is zero. Therefore this gives ${}^{\sm} C^{\infty} (X)$ an LTVS topology.

Notice that each element in $G$ acts on ${}^{\sm} C^{\infty} (X)$ in a continuous way.

\subsubsection{}

Let us denote by $\calD_{\pc} (X)$ the right $G$-module of continuous functionals on ${}^{\sm} C^{\infty} (X)$ (here ``$\pc$" stands for ``pro-compactly supported"). We have: $$ \calD_{\pc} (X) = \lim_{K \subset G} \ \calD_{c} (X/K)$$ where $\calD_c (X/K)$ denotes the space of distributions with compact support on the discrete space $X/K$. \quash{Given $\delta \in \calD_{\pc} (X)$ and $\phi \in {}^{\sm} C^{\infty} (X)$, one can suggestively write $\int_X \phi \cdot \delta$ for the evaluation $\delta(\phi)$.}

\subsubsection{}

Let us denote by $$\calA (X) \subset {}^{\sm} C^{\infty} (X)$$ the subspace consisting of functions $\phi$ for which the $G$-submodule spanned by $\phi$ is of finite length. Then $\calA (X)$ is a $G$-submodule of ${}^{\sm} C^{\infty} (X)$.

\subsection{Spherical spaces}

\subsubsection{}

We say that $X$ is \textbf{$F$-spherical} if given a minimal parabolic $\mathbb{P} \subset \mathbb{G}$, $\mathbb{P} (F)$ has finitely many orbits on $X$.

\subsubsection{}

We say that $X$ \textbf{has the finite multiplicity property} if for every $G$-equivariant local system $\calL$ of rank $1$ on $X$ and every $V \in \calM (G)$ of finite length, the space $\Hom_G (V , \Gamma (X , \calL))$ is finite-dimensional. It is enough to check this for irreducible $V \in \calM (G)$. By Frobenius reciprocity, choosing $x_0 \in X$ and setting $H := \textnormal{Stab}_G (x_0)$ (so that $X \cong H \backslash G$), $X$ has the finite multiplicity property if and only if for every $V \in \calM (G)$ of finite length (or irreducible), and every smooth character $\chi : H \to \bbC^{\times}$, the space of $(H,\chi)$-equivariant functionals $\Hom_H (V , \bbC_{\chi})$ is finite-dimensional. This is the same as requiring that the space of $\chi$-twisted $H$-coinvariants $$ {}_{(H,\chi)} V := V / \langle hv - \chi (h) v\rangle_{h \in H , v \in V}$$ is finite-dimensional.

\subsubsection{}

It is conjectured that if $X$ is $F$-spherical then $X$ has the finite multiplicity property, and this conjecture is known in a large amount of cases. See for example \cite[Theorem 5.1.5]{SaVe}, \cite[Theorem 4.5]{De}.

\subsection{The main theorem}

\subsubsection{}

We now state the main theorem of the article.

\begin{theorem}\label{thm main theorem}
    Let $X$ be a homogeneous right $G$-space which is $F$-spherical and has the finite multiplicity property. Then the inclusion $\calA (X) \subset {}^{\sm} C^{\infty}(X)$ is Paley-Wiener (in the sense of \S\ref{sssec def of Paley-Wiener}).
\end{theorem}

\subsubsection{}\label{sssec iota}

Let us unfold the statement of Theorem \ref{thm main theorem}. Consider the linear map $$ \iota_X : \calD_{\pc} (X) \to \calA (X)^*$$ given by sending a distribution $\delta$, which is a linear functional on ${}^{\sm} C^{\infty} (X)$, to its restriction to $\calA(X)$. The theorem says that $\iota_X$ is injective, and characterizes its image. Let us spell out this characterization. Let $Z \in \Aff$ be smooth. Given a function $$\Phi : Z(\bbC) \times X \to \bbC,$$ denote $\Phi^z (x) = \Phi (z,x)$ and $\Phi_x (z) = \Phi (z,x)$. We will first say what it means for $\Phi$ to be an algebraic mapping (this corresponds to $z \mapsto \Phi^z$ being an algebraic mapping from $Z$ to $\calA (X)$ in the sense of \S\ref{sssec alg map for reduced}). The conditions are:

\begin{itemize}
    \item For every $z \in Z(\bbC)$ one has $\Phi^z \in \calA (X)$ and for every $x \in X$ one has $\Phi_x \in \calO_Z$.
	\item There exists a open compact subgroup $K \subset G$ such that for every $z \in Z(\bbC)$, $\Phi^z \in {}^K \calA (X)$.
    \item Given an algebraic differential operator $D$ on $Z$, denote by $$D\Phi : Z(\bbC) \times X \to \bbC$$ the function for which $(D\Phi)_x = D(\Phi_x)$. Then for every algebraic differential operator $D$ on $Z$ and every $z \in Z(\bbC)$, one has $(D\Phi)^z \in \calA (X)$.
\end{itemize}

Then the condition on a functional $\ell \in \calA (X)^*$, Theorem \ref{thm main theorem} guarantees, to be in the image of $\iota_X$ is that for every smooth $Z \in \Aff$ and every algebraic mapping $\Phi : Z(\bbC) \times X \to \bbC$, the following holds:

\begin{itemize}
	\item The function on $Z(\bbC)$ given by $z \mapsto \ell (\Phi^z)$ is algebraic - denote it $f_{\Phi} \in \calO_Z$.
	\item Given an algebraic differential operator $D$ on $Z$, we have $(D f_{\Phi} )(z) = \ell((D\Phi)^z)$.
\end{itemize}

\subsubsection{}

To prove Theorem \ref{thm main theorem}, we will recast it in a more categorical form in order to use structure theory of $\calM (G)$ \`a la Bernstein. In \S\ref{sec categorical} we will present the general categorical framework, and in \S\ref{sec proof of main thm} we will prove Theorem \ref{thm main theorem}.

%%%%%%%%%%%%%%%%%%%%%%%%%%%%%%%%%%%%
%%%%%%%%%%%%%%%%%%%%%%%%%%%%%%%%%%%%

\section{A categorical Paley-Wiener framework}\label{sec categorical}

In this section we provide a categorical framework into which we will recast Theorem \ref{thm main theorem}. Throughout this section, we fix a $\bbC$-linear abelian category $\calM$ admitting all small colimits. By $\calM^{\fl} \subset \calM$ we denote the full subcategory consisting of objects of finite length. By $\Fun^{l} (\calM , \Vect)$ we denote the category of functors, from $\calM$ to the category of vector spaces $\Vect$, which commute with all small colimits (``$l$" stands for ``left adjoint", alluding to the fact that, under suitable presentability conditions on $\calM$, a functor commutes with all small colimits if and only if it is a left adjoint, i.e. has a right adjoint).

\subsection{Objects with algebra action and their punctual finite length properties}

\subsubsection{}

Given $A \in \Alg$, we denote by $\calM^A$ the abelian category whose objects are pairs $(M,\xi)$ consisting of an object $M \in \calM$ and an algebra morphism $\xi : A \to End(M)$ (morphisms are defined in a straightforward way). For example, $\Vect^A$ is the category of (left) $A$-modules.

\subsubsection{}

We will denote by $M \mapsto \underline{M}$ the forgetful functor $\calM^A \to \calM$.

\subsubsection{}

Notice that a functor $C \in \Fun^{l} (\calM , \Vect)$ automatically upgrades to a functor $C \in \Fun (\calM^A , \Vect^A)$ (which by abuse of notation we will denote by the same letter).

\subsubsection{}

Given a morphism $A \to B$ in $\Alg$, we have an adjunction $$ B \underset{A}{\otimes} - : \ \calM^A \rightleftarrows \calM^B : \ (-)^B_A,$$ where $(-)^B_A$ is the forgetful functor (i.e. $M^B_A$ is $M$ equipped with the algebra morphism $A \to B \to \End (M)$). Given an ideal $I \subset A$, let us also abbreviate $M/IM := (A/I) \underset{A}{\otimes} M$.

\subsubsection{}

Given a morphism $A \to B$ in $\Alg$ and $C \in \Fun^{l} (\calM , \Vect)$, we have naturally $$ C(B \underset{A}{\otimes} -) \cong B \underset{A}{\otimes} C(-).$$

\subsubsection{}

Let $A \in \Alg$ and let $M \in \calM^A$. We will say that $M$ is \textbf{punctually of finite length} if, for every cofinite ideal $I \subset A$, the object $\underline{M/IM} \in \calM$ has finite length. We will say that $M$ is \textbf{regularly punctually of finite length} if it is punctually of finite length and for every morphism $\underline{M} \to V$, where $V$ has finite length, there exists a cofinite ideal $I \subset A$ such that $\underline{M} \to V$ factors via the canonical $\underline{M} \to \underline{M/IM}$.

\subsection{Compact projective objects and their properties}

\subsubsection{}

Given a projective object $P \in \calM$, we have an adjunction $$ - \underset{\End (P)}{\otimes} P \ : \calM (\End(P)^{\op}) \rightleftarrows \calM : \ \Hom (P,-).$$ The left adjoint is fully faithful, and its essential image consists of objects $M \in \calM$ admitting a resolution $$ P^{\oplus I} \to P^{\oplus J} \to M \to 0.$$

\subsubsection{}

A projective object $P \in \calM$ is called \textbf{compact} if $$\Hom (P,-) : \calM \to \Vect$$ commutes with filtered small colimits, or, equivalently, with small direct sums.

\subsubsection{}

We say that there are \textbf{enough} compact projective objects in $\calM$ with some property $\pi$ if there exists a set $\calP$ of compact projective objects in $\calM$ with property $\pi$ such that for every $0 \neq M \in \calM$ there exists an object $P \in \calP$ such that $\Hom (P,M) \neq 0$.

\subsubsection{}

It is not hard to see that if there are enough compact projective objects in $\calM$ with a property $\pi$, then for every compact projective object $P \in \calM$ there exist compact projective objects $P_1 , \ldots , P_k \in \calM$ with property $\pi$ such that $P$ is isomorphic to a direct summand in $P_1 \oplus \ldots \oplus P_k$.

\subsubsection{}

We will say that a compact projective object $P \in \calM$ is \textbf{splitting} if the essential image of the functor $- \underset{\End (P)}{\otimes} P$ above is a direct summand of $\calM$.

\subsubsection{}

Let us say that a compact projective object $P \in \calM$ is \textbf{familial} if the center of $\End (P)$ is a finitely generated $\bbC$-algebra, and $\End (P)$ is finitely generated as a module over its center.

\subsubsection{} The following two claims will serve us in our case of interest, to check that properties related to compact projective objects are satisfied.

\begin{claim}\label{clm criterion rpfl}
    Let $P \in \calM$ be a splitting familial compact projective object. Denote by $A$ the center of $\End (P)$. Then considering $P$ as an object of $\calM^{A}$, it is regularly punctually of finite length.
\end{claim}

\begin{proof}
	Given in \S\ref{ssec proof of clm criterion rpfl}.
\end{proof}

\begin{claim}\label{clm criterion finite generation}
    Let $C \in \Fun^{l} (\calM , \Vect)$. Suppose that there are enough compact projective objects in $\calM$ which are splitting and familial, and that there are enough compact projective objects $P$ in $\calM$ for which $C(P)$ is a finitely generated $\End (P)$-module. Then $C(P)$ is a finitely generated $\End (P)$-module for every compact projective object $P \in \calM$.
\end{claim}

\begin{proof}
	Given in \S\ref{ssec proof of clm criterion finite generation}.
\end{proof}

\subsection{Algebraic morphsims and the categorical Paley-Wiener proposition}

\subsubsection{}

By a \textbf{pair} we will mean a pair $(Z,M)$ consisting of $Z \in \Aff$ and $M \in \calM^{\calO_Z}$ which is punctually of finite length. A family $\pi$ of pairs will be said to be \textbf{stable under finite morphisms} if given a pair $(Z,M)$ in $\pi$ and a finite morphism $Z \to W$, the resulting pair $(W , M^{\calO_Z}_{\calO_W})$ also lies in $\pi$. A family $\pi$ of pairs will be said to be \textbf{stable under base change} if given a pair $(Z,M)$ in $\pi$ and a morphism $W \to Z$, the resulting pair $(W , \calO_W \underset{\calO_Z}{\otimes} M)$ also lies in $\pi$.

\subsubsection{}\label{sssec algebraic morphism}

Let $D,C \in \Fun^{l} (\calM , \Vect)$ and let $t : D|_{\calM^{\fl}} \to C|_{\calM^{\fl}}$. Let $(Z,M)$ be a pair.

\medskip

We say that the data $(t,Z,M,m,\alpha)$, where $m \in M$ and $\alpha \in \Hom_{\calO_Z} (C(M) , \calO_Z)$, is \textbf{algebraic} if there exists $a \in \calO_Z$ such that for every zero-dimensional closed subscheme $W \subset Z$, the image of $d$ under $$ D(M) \to D(\calO_W \underset{\calO_Z}{\otimes} M) \xrightarrow{t_{\calO_W \underset{\calO_Z}{\otimes} M}} C(\calO_W \underset{\calO_Z}{\otimes} M) \cong \calO_W \underset{\calO_Z}{\otimes} C(M) \xrightarrow{\alpha} \calO_W$$ is equal to the image of $a$ under the projection $\calO_Z \to \calO_W$.

\medskip

Clearly, if $Z$ is zero-dimensional then every data $(t,Z,M,m,\alpha)$ is algebraic.

\medskip

Assuming for simplicity that $Z$ is irreducible, we say that the data $(t,Z,M,d,\alpha)$ is \textbf{generically algebraic} if there exists a dense distinguished open subset $U \to Z$ such that the obviously resulting data $(t,U,\calO_U \underset{\calO_Z}{\otimes} M , [d] , [\alpha])$ is algebraic.

\medskip

We say that the data $(t,Z,M)$ is \textbf{algebraic} (resp. \textbf{generically algebraic}) if $(t,Z,M,d,\alpha)$ is algebraic (resp. generically algebraic) for all $d \in D(M)$ and $\alpha \in \Hom_{\calO_Z} (C(M) , \calO_Z)$.

\medskip

We say that $t$ is \textbf{algebraic} if the data $(t,Z,M)$ is algebraic for all pairs $(Z , M)$.

\medskip

Given a family $\pi$ of pairs closed under finite morphisms and closed under base change, let us say that $t$ is \textbf{$\pi$-$\bbA$-almost algebraic} if for any pair $(\bbA^n , M)$ in $\pi$ the data $(t , \bbA^n , M)$ is algebraic if $n = 1$ and generically algebraic if $n > 1$.

\subsubsection{}

We denote by $$ \Hom^{\alg} (D|_{\calM^{\fl}} , C|_{\calM^{\fl}}) \subset \Hom^{\pi\textnormal{-}\bbA\textnormal{-}\textnormal{a }\alg} (D|_{\calM^{\fl}} , C|_{\calM^{\fl}}) \subset \Hom (D|_{\calM^{\fl}} , C|_{\calM^{\fl}})$$ the subspaces of algebraic and $\pi$-$\bbA$-almost algebraic elements. Restriction defines a map $$ \Hom (D,C) \to \Hom^{\alg} (D|_{\calM^{\fl}} , C|_{\calM^{\fl}}).$$

\subsubsection{}

The main proposition of the current section is the following.

\begin{proposition}[categorical Paley-Wiener]\label{prop abstract setting}
	Let $C \in \Fun^{l} (\calM , \Vect)$ and let $\pi$ be a family of pairs stable under finite morphisms and stable under base change. Suppose that there are enough compact projective objects in $\calM$ of the form $\underline{M}$ where $(Z,M)$ is a pair in $\pi$ such that, in addition, $M$ is regularly punctually of finite length and  $C(\underline{M})$ is a finitely generated $\calO_Z$-module. Then for any $D \in \Fun^{l}(\calM , \Vect)$ the restriction map $$ \Hom (D,C) \to \Hom^{\pi\textnormal{-}\bbA\textnormal{-}\textnormal{a }\alg} (D|_{\calM^{\fl}} , C|_{\calM^{\fl}})$$ is an isomorphism. In particular, we have an equality $$ \Hom^{\alg} (D|_{\calM^{\fl}} , C|_{\calM^{\fl}}) = \Hom^{\pi\textnormal{-}\bbA\textnormal{-}\textnormal{a }\alg} (D|_{\calM^{\fl}} , C|_{\calM^{\fl}})$$ and the restriction map $$ \Hom (D,C) \to \Hom^{\alg} (D|_{\calM^{\fl}} , C|_{\calM^{\fl}})$$ is an isomorphism.
\end{proposition}

\begin{proof}
	Given in \S\ref{ssec proof of prop abstract setting}.
\end{proof}

In view of Claim \ref{clm criterion rpfl} and Claim \ref{clm criterion finite generation} we can also formulate the following corollary:

\begin{corollary}\label{cor prop abstract setting}
	Let $C \in \Fun^{l} (\calM , \Vect)$ and let $\pi$ be a family of pairs stable under finite morphisms and stable under base change. Suppose that the following conditions hold:
		\begin{enumerate}
		    \item There are enough compact projective objects $P$ in $\calM$ which are splitting and familial, and such that the pair $(Spec(Z(\End (P))),P)$ lies in $\pi$.
		    \item There are enough compact projective objects $P$ in $\calM$ for which $C(P)$ is a finitely generated $\End(P)$-module.
		\end{enumerate}
		
		Then for any $D \in \Fun^{l}(\calM , \Vect)$ the restriction map $$ \Hom (D,C) \to \Hom^{\pi\textnormal{-}\bbA\textnormal{-}\textnormal{a }\alg} (D|_{\calM^{\fl}} , C|_{\calM^{\fl}})$$ is an isomorphism. In particular, we have an equality $$ \Hom^{\alg} (D|_{\calM^{\fl}} , C|_{\calM^{\fl}}) = \Hom^{\pi\textnormal{-}\bbA\textnormal{-}\textnormal{a }\alg} (D|_{\calM^{\fl}} , C|_{\calM^{\fl}})$$ and the restriction map $$ \Hom (D,C) \to \Hom^{\alg} (D|_{\calM^{\fl}} , C|_{\calM^{\fl}})$$ is an isomorphism.
\end{corollary}

Of course, we can further simplify by taking $\pi$ to be the family of all pairs.

%%%%%%%%%%%%%%%%%%%%%%%%%%%%%%%%%%%%
%%%%%%%%%%%%%%%%%%%%%%%%%%%%%%%%%%%%

\section{Proof of the main theorem (Theorem \ref{thm main theorem})}\label{sec proof of main thm}

In this section we prove Theorem \ref{thm main theorem}. Throughout this section, we fix an $F$-spherical homogeneous right $G$-space $X$, having the finite multiplicity property. We fix $x_0 \in X$ and denote $H := Stab_G (x_0)$, so that $X \cong H \backslash G$.

\subsection{Function-to-functor isomorphisms}

\subsubsection{}

Let $Forg \in \Fun^{l} (\calM (G) , \Vect)$ be the forgetful functor, and let $Coinv_H \in \Fun^{l} (\calM (G) , \Vect)$ be the functor of $H$-coinvariants: $$ Coinv_H (M) := {}_H M := M / \langle hm - m\rangle_{h \in H , m\in M}.$$

\subsubsection{}

We will now describe an isomorphism of right $G$-modules \begin{equation}\label{eq functor hom description} \calD_{\pc} (H \backslash G) \cong \Hom (Forg , Coinv_H)\end{equation} (for which we don't need $H \backslash G$ to be neither $F$-spherical, nor to have the finite multiplicity property). The right $G$-module structure on the space of functors is by the left $G$-module structure on the forgetful functor.

\medskip

For a right $G$-module $V$, let us denote by $V^{\wedge}$ the limit of spaces $V^K$ for $K \subset G$ an open compact subgroup, with respect to the maps $V^K \to V^L$ given by averaging with respect to $L$ on the right, whenever $K \subset L$. Then for any $G$-module $V$ we have $$ \Hom_{\calH_G^{\op}} (\calH_G , V) \cong V^{\wedge}$$ given by sending the morphism $S$ to the element which is equal to $S(e_K)$ at the $K$-component, where $e_K$ is the idempotent corresponding to $K$. On the other hand, for any $C \in \Fun^{l} (\calM (G) , \Vect)$ it is standard to construct an isomorphism $$ \Hom (Forg , C) \cong \Hom_{\calH_G^{\op}} (\calH_G , C(\calH_G))$$ (here $C(\calH_G)$ has a right $G$-module structure since $\calH_G$ has one). Combining the two, we obtain an isomorphism $$ \Hom (Forg , C) \cong C(\calH_G)^{\wedge}.$$ In our special case $C = Coinv_H$ we obtain $$ \Hom (Forg, Coinv_H) \cong ({}_H (\calH_G))^{\wedge}.$$ The standard pushforward map (given by integrating along $H$) $$ \calH_G = \calD_c (G)^{\sm} \to \calD_c (H \backslash G)^{\sm}$$ sets an isomorphism $$ {}_H (\calH_G) \cong \calD_c (H \backslash G)^{\sm}.$$ Therefore we obtain $$ ({}_H (\calH_G))^{\wedge} \cong (\calD_c (H \backslash G)^{\sm})^{\wedge} \cong \calD_{\pc} (H \backslash G).$$ Collecting these, we obtain isomorphism (\ref{eq functor hom description}).

\subsubsection{}

Next we describe an isomorphism of right $G$-modules \begin{equation}\label{eq isom fnctnl fnctr} \calA(H \backslash G)^* \cong \Hom (Forg|_{\calM (G)^{\fl}} , Coinv_H|_{\calM (G)^{\fl}}),\end{equation}

\medskip

We will describe the construction of isomorphism (\ref{eq isom fnctnl fnctr}) in both directions; To verify all details is then routine.

\medskip

Let $\ell \in \calA(H \backslash G)^*$ and $V \in \calM (G)^{\fl}$. Since $H\backslash G$ has the finite multiplicity property, ${}_H V \cong (({}_H V)^*)^* \cong (V^{*,H})^*$. Therefore, to define the desired map $V \to {}_H V$ is the same as to define a pairing $$ V \otimes V^{*,H}\to \bbC.$$ Thus, given $v \in V$ and $\alpha \in V^{*,H}$, we would like to obtain a number. We get this number by applying $\ell$ to the function in $\calA (H \backslash G)$ given by $$Hg \mapsto \alpha (g v).$$

\medskip

Conversely, let us be given a morphism of functors $$ t : Forg|_{\calM(G)^{\fl}} \to Coinv_H|_{\calM(G)^{\fl}}.$$ To construct a corresponding $\ell$, given $\phi \in \calA(H \backslash G)$ we want to obtain a number. Let $V \subset \calA (H \backslash G)$ be the $G$-submodule generated by $\phi$. By definition, it has finite length. Therefore we have a linear map $t_V : V \to {}_H V$. Since $H \backslash G$ has the finite multiplicity property, we have ${}_H V \cong (({}_H V)^*)^* \cong (V^{*,H})^*$. Therefore, $t_V$ defines a pairing $$ V \otimes V^{*,H} \to \bbC.$$ As we have the element $\phi \in V$ and the element in $V^{*,H}$ given by the embedding $V \to \calA(H \backslash G)$ composed with the evaluation at the origin $H$ of $H \backslash G$, we obtain a number, as desired.

\subsection{The categorical Paley-Wiener theorem put in action}

\subsubsection{}

We would like to use Corollary \ref{cor prop abstract setting} in our setting. So we need to check the conditions:

\begin{claim}\label{clm M(G) and Coinv_H are suitable}\
	\begin{enumerate}
		\item There are enough compact projective objects $P$ in $\calM (G)$ which are splitting and familial and such that for every open compact subgroup $K \subset G$, ${}^K P$ is a finitely generated module over the center of $\End (P)$.
		\item There are enough compact projective objects $P$ in $\calM (G)$ for which $Coinv_H (P)$ is a finitely generated $\End (P)$-module.
	\end{enumerate}
\end{claim}

\begin{proof}
	Given in \S\ref{ssec proof of clm M(G) and Coinv_H are suitable}.
\end{proof}

Therefore, applying Corollary \ref{cor prop abstract setting}, we obtain:

\begin{corollary}\label{cor cat pal wie for Coinv}
	Let $\pi$ be the property of pairs $(Z , M)$ (where $Z \in \Aff$ and $M \in \calM (G)^{\calO_Z}$ is punctually of finite length) which is that, for every open compact subgroup $K \subset G$, ${}^K M$ is finitely generated over $\calO_Z$. Then, for every $D \in \Fun^l (\calM (G) , \Vect)$, the restriction map $$ \Hom (D,Coinv_H) \xrightarrow{} \Hom^{\pi\textnormal{-}\bbA\textnormal{-a }\alg} (D|_{\calM (G)^{\fl}} , Coinv_H|_{\calM(G)^{\fl}})$$ is an isomorphism and the inclusion $$ \Hom^{\alg} (D|_{\calM (G)^{\fl}} , Coinv_H|_{\calM(G)^{\fl}}) \subset \Hom^{\pi\textnormal{-}\bbA\textnormal{-a }\alg} (D|_{\calM (G)^{\fl}} , Coinv_H|_{\calM(G)^{\fl}})$$ is an equality.
\end{corollary}

\begin{remark}
	The additional property $\pi$ of the ${}^K M$'s being finitely generated over $\calO_Z$ we don't need here, but we will need it in \S\ref{sec Bernstein morphisms}.
\end{remark}

\subsubsection{}

\subsection{The proof of Theorem \ref{thm main theorem}}

\subsubsection{}

We have a commutative diagram \begin{equation}\label{eq comm diagram} \xymatrix{\calD_{\pc}(H \backslash G) \ar@{<->}[r]^(0.42){\sim} \ar[d]^{\iota_{H \backslash G}} & \Hom (Forg, Coinv_H) \ar[d]^{\textnormal{restriction}} \\ \calA(H \backslash G)^* \ar@{<->}[r]^(0.3){\sim} & \Hom (Forg|_{\calM(G)^{\fl}} , Coinv_H|_{\calM(G)^{\fl}}) } \end{equation} where the upper horizontal isomorphism is (\ref{eq functor hom description}) and the lower horizontal isomorphism is (\ref{eq isom fnctnl fnctr}). The map $\iota_{H \backslash G}$ is the natural map, considered in \S\ref{sssec iota}. The commutativity of this diagram follows from unpacking all definitions.

\subsubsection{}

The image of the restriction map in diagram (\ref{eq comm diagram}) lies in the subspace of algebraic morphisms. We claim that under the lower horizontal isomorphism in diagram (\ref{eq comm diagram}), functionals on the left that are algebraic (in the sense of \S\ref{sssec algebraic functional}) will map to morphisms on the right that are algebraic.

\medskip

Indeed, assume that $\ell \in \calA(H \backslash G)^*$ is algebraic. Let us be given a smooth $Z \in \Aff$, an $M \in \calM(G)^{\calO_Z}$ punctually of finite length, an $m \in M$ and a morphism of $\calO_Z$-modules $\alpha : {}_H M \to \calO_Z$ (notice that here we can assume that $Z$ is smooth by Corollary \ref{cor cat pal wie for Coinv}). Tracing the definitions, we see that what we need to show is that there exists $a \in \calO_Z$ such that, for every $\lambda \in \PD(Z)$, $\lambda (a)$ equals to $\ell$ applied to the function in $\calA(H \backslash G)$ given by $$ Hg \mapsto \lambda(\alpha(gm)).$$ In other words, we need to see that the element of $\PD(Z)^*$ given by $$ \lambda \mapsto \ell \left[ Hg \mapsto \lambda(\alpha(gm)) \right]$$ is algebraic. Since $\ell$ is assumed to be algebraic, this will be so if the linear map $$\PD(Z) \to \calA(H\backslash G)$$ given by $$ \lambda \mapsto \left[ Hg \mapsto \lambda (\alpha(gm)) \right]$$ is an algebraic mapping from $Z$ to $\calA (H \backslash G)$, and this is straightforward.

\subsubsection{}

We thus obtained a commutative diagram $$ \xymatrix{\calD_{\pc}(H \backslash G) \ar@{<->}[r]^(0.42){\sim} \ar[d]^{\iota_{H \backslash G}} & \Hom (Forg, Coinv_H) \ar[d]^{\textnormal{restriction}} \\ \calA(H \backslash G)^{*,\alg} \ar@{^{(}->}[r] & \Hom^{\alg} (Forg|_{\calM(G)^{\fl}} , Coinv_H|_{\calM(G)^{\fl}}) }.$$ Since we have seen in Corollary \ref{cor cat pal wie for Coinv} that the restriction map on the right is a bijection, we obtain that $\iota_{H \bs G}$ is a bijection, establishing Theorem \ref{thm main theorem}.

%%%%%%%%%%%%%%%%%%%%%%%%%%%%%%%%%%%%
%%%%%%%%%%%%%%%%%%%%%%%%%%%%%%%%%%%%

\section{The construction of Bernstein morphisms}\label{sec Bernstein morphisms}

In this section, we discuss the construction of Bernstein morphisms using the approach of this article. Throughout this section, we fix a closed subgroup $H \subset G$ such that the homogeneous right $G$-space $H \backslash G$ is $F$-spherical. We will also, starting from \S\ref{ssec morphism b tilde}, fix a p-pair $(\bbP , \bbL)$ suitable for degenerating $H$ (which we define in \S\ref{sssec ppair suitable}).

\subsection{Degeneration of $H$}

\subsubsection{}

A \textbf{p-pair} is the data $(\bbP, \bbL)$ of a parabolic $\bbP \subset \bbG$ and a Levi subgroup $\bbL \subset \bbP$. We denote by $\bbP^-$ the parabolic opposite to $\bbP$ such that $\bbP^- \cap \bbP = \bbL$, and by $\bbU$ and $\bbU^-$ the unipotent radicals of $\bbP$ and $\bbP^-$, respectively. We denote then by $T$ the $F$-points of the maximal split torus in the center of $\bbL$, and by $T^+ \subset T$ (resp. $T^{++} \subset T$) the subset consisting of $\gamma \in T$ for which $\{ \gamma^{-n} u \gamma^n : n \in \bbZ_{\ge 0} \} $ is relatively compact in $\bbU (F)$ (resp. $\gamma^{-n} u \gamma^n \underset{n \to \infty}{\to} 1$) for any $u \in \bbU (F)$. We say that a property of elements of $T^{++}$ holds \textbf{deep enough} if there exists neighborhoods $V_1 , V_2$ of $1$ in $\bbU (F)$ such that the property holds for all $\gamma \in T^{++}$ satisfying $\gamma^{-1} V_1 \gamma \subset V_2 $. As is well known, there are arbitrarily small open compact subgroups $K \subset G$ satisfying Bruhat's theorem with repsect to the p-pair $(\bbP , \bbL)$ (see \cite[\S 2.1]{BeRu}) and in particular satisfying $e_K \gamma e_K \cdot e_K \delta e_K = e_K \gamma \delta e_K$ for $\gamma , \delta \in T^{+}$ (see \cite[\S 2.1, Lemma 20]{BeRu}).

\subsubsection{}\label{sssec ppair suitable}

We will say that the p-pair $(\bbP , \bbL)$ is \textbf{suitable for degenerating $H$} if the following holds:

\begin{itemize}
	\item Given an open compact subgroup $K \subset G$, there exists a neighborhood of $1$ in $G$ such that for deep enough $\gamma \in T^{++}$ we have $p (g \gamma  e_K) = p (\gamma e_K)$ for all $g$ in this neighborhood, where $p : \calH_G \to {}_H (\calH_G)$ is the projection onto $H$-coinvariants.
\end{itemize}

It is easy to verify that in fact this notion only depends on $\bbP$, i.e. if $(\bbP , \bbL)$ is suitable for degenerating $H$ and if $\bbL^{\prime} \subset \bbP$ is a Levi subgroup then $(\bbP , \bbL^{\prime})$ is also suitable for degenerating $H$.

\subsubsection{}

Let us describe how a p-pair $(\bbP , \bbL)$ suitable for degenerating $H$ can be obtained. Let $\bbP \subset \bbG$ be a parabolic such that $x_0 \cdot \bbP (F)$ is open in $X \cong H \backslash G$, i.e. $H \cdot \bbP (F)$ is open on $G$. Let $\bbL \subset \bbP$ be any Levi subgroup. We claim that $(\bbP , \bbL)$ is suitable for degenerating $H$. Indeed, we can find a neighborhood $V$ of $1$ in $\bbP (F)$ such that $\gamma^{-1} V \gamma \subset V$ for all $\gamma\in T^{++}$ and such that $V \subset K$. Then for $g$ lying in $H V$ (notice that $HV$ is open in $G$ since the orbit map $\bbP (F) \to H \backslash G$ via $x_0$ has open image and hence is open) we have $g \in Hv$ for some $v \in V$ and so for $\gamma \in T^{++}$ we have $$p(g \gamma e_K) = p(v \gamma e_K) = p (\gamma (\gamma^{-1} v \gamma) e_K) = p(\gamma e_K)$$ since $\gamma^{-1} v \gamma \in K$.

\subsubsection{}

Given a p-pair $(\bbP , \bbL)$ suitable for degenerating $H$, we define a closed subgroup $H_{\bbP , \bbL} \subset G$ (the \textbf{degenerated subgroup}) as consisting of the elements $g \in G$ which satisfy: 

\begin{itemize}
	\item Given an open compact subgroup $K \subset G$, for deep enough $\gamma \in T^{++}$ we have $p (\gamma g e_K) = p (\gamma e_K)$ where again $p : \calH_G \to {}_H (\calH_G)$ is the projection onto $H$-coinvariants.
\end{itemize}

Let us check that indeed $H_{\bbP , \bbL}$ is a closed subgroup of $G$. It is clear that $1$ belongs to $H_{\bbP , \bbL}$. Given $g_1 , g_2 \in H_{\bbP , \bbL}$ and an open compact subgroup $K \subset G$ we have $$p(\gamma g_1 g_2 e_K) = p(\gamma g_1 e_{g_2 K g_2^{-1}} g_2) = \ldots $$ and or deep enough $\gamma \in T^{++}$ we continue $$ \ldots = p (\gamma  e_{g_2 K g_2^{-1}} g_2) = p( \gamma g_2 e_K) \ldots$$ and for deep enough $\gamma \in T^{++}$ we get $$ \ldots = p (\gamma e_K).$$ To show that $H_{\bbP , \bbL}$ is closed, let $g$ be in the closure of $H_{\bbP , \bbL}$. Let $K \subset G$ be an open compact subgroup. Then there exists $g^{\prime} \in H_{\bbP , \bbL}$ such that $g^{\prime} K =g  K$. Then $p (\gamma g e_K) = p(\gamma g^{\prime} e_K)$ and so clearly $g \in H_{\bbP , \bbL}$ since $g^{\prime} \in H_{\bbP , \bbL}$.

\subsection{The morphism $\widetilde{b}_{H,\bbP , \bbL}$}\label{ssec morphism b tilde}

\subsubsection{}

We will now, following Casselman's classical ``canonical pairing" construction, define a morphism of functors $$ \widetilde{b}_{H,\bbP , \bbL} : Forg|_{\calM (G)^{\fl}} \to Coinv_H|_{\calM(G)^{\fl}}.$$

\subsubsection{}

Given $V \in \calM (G)^{\fl}$ we need to define a linear map $V \to {}_H V$. In view of the finite multiplicity property of $H$, this is the same as defining a bilinear pairing $V \otimes V^{*,H} \to \bbC$. Thus, given $v \in V$ and $\alpha \in V^{* , H}$ we need to define a number. To do this, we first fix an auxiliary $\gamma \in T^{++}$, and then consider the sequence $c^{v,\alpha , \gamma}_{\bullet} \in \bbC^{\bbZ_{\ge 0}}$ given by $$ c^{v,\alpha , \gamma}_n := \alpha (\gamma^n v).$$

\subsubsection{}\label{sssec choosing L and K}

Consider the shift operator $$ \textnormal{Shift} : \bbC^{\bbZ_{\ge 0}} \to \bbC^{\bbZ_{\ge 0}}, \quad \textnormal{Shift} ((c_n)_{n \in \bbZ_{\ge 0}}) := (c_{n+1})_{n \in \bbZ_{\ge 0}}.$$ We claim that our $c_{\bullet}^{v,\alpha , \gamma}$ is \textnormal{Shift}-finite (i.e. the smallest subspace of $\bbC^{\bbZ_{\ge 0}}$ containing $c_{\bullet}^{v,\alpha , \gamma}$ and closed under the operator \textnormal{Shift} is finite-dimensional). Indeed, we choose an open compact subgroup $L \subset G$ such that $v \in {}^L V$. Then since $(\bbP , \bbL)$ is suitable for degenerating $H$ there exists an open compact subgroup $K \subset G$ such that $p (k \delta e_L) = p (\delta e_L)$ for $k \in K$ and deep enough $\delta \in T^{++}$, so $\alpha (k \delta v) = \alpha (\delta v)$ for $k \in K$ and deep enough $\delta \in T^{++}$. We can, by making $K$ smaller, assume that $K \subset L$ and also that $K$ satisfies Bruhat's theorem with respect to the p-pair $(\bbP , \bbL)$. Then $v \in {}^K V$ and $\alpha ( (e_K \delta e_K) v) = \alpha (\delta v)$ for deep enough $\delta \in T^{++}$. Therefore, for big enough $n \in \bbZ_{\ge 0}$ we have $$c_n^{v,\alpha , \gamma} = \alpha (\gamma^n v) = \alpha ( (e_K \gamma^n e_K) v)=  \alpha ((e_K \gamma e_K)^n v).$$ Therefore, it is enough to see that $( n \mapsto \alpha ((e_K \gamma e_K)^n v))$ is $\textnormal{Shift}$-finite, as the difference between $c_{\bullet}^{v,\alpha,\gamma}$ and it is eventually zero, and thus annihilated by some power of $\textnormal{Shift}$. To that end, we consider the $\bbC$-linear map $\mu : {}^K V \to \bbC^{\bbZ_{\ge 0}}$ sending $v^{\prime}$ to $(n \mapsto \alpha ( (e_K \gamma e_K)^n v^{\prime})$. The map $\mu$ intertwines $e_K \gamma e_K$ and $\textnormal{Shift}$, and therefore, since ${}^K V$ is finite-dimensional, the image of $\mu$ consists of $\textnormal{Shift}$-finite vectors.

\subsubsection{}

For a $\textnormal{Shift}$-finite $c_{\bullet} \in \bbC^{\bbZ_{\ge 0}}$ we can uniquely write $$c_{\bullet} = {}^{\textnormal{nilp}} c_{\bullet} + {}^{\textnormal{inv}} c_{\bullet}$$ where ${}^{\textnormal{nilp}} c_{\bullet}$ (resp. ${}^{\textnormal{inv}} c_{\bullet}$ ) is contained in a finite-dimensional $\textnormal{Shift}$-closed subspace all of the generalized $\textnormal{Shift}$-eigenvalues on which are zero (resp. non-zero).

\subsubsection{}

We now define the desired number, the result of the sought-for pairing $V \otimes V^{*,H} \to \bbC$ applied to $v \otimes \alpha$, as $({}^{\textnormal{inv}} c^{v,\alpha , \gamma})_0$.

\subsubsection{}\label{sssec second description of inv}

We want here to check that the number defined does not depend on the choice of $\gamma$. Indeed, decompose ${}^K V = W_1 \oplus W_2$ where $W_1$ and $W_2$ are $e_K \gamma e_K$-invariant, and all the generalized eigenvalues of $e_K \gamma e_K$ on $W_1$ (resp. $W_2$) are zero (resp. non-zero). Write $v = v_1 + v_2$ where $v_1 \in W_1$ and $v_2 \in W_2$. It is clear that $\mu (v_2) = {}^{\textnormal{inv}} \mu(v) = ({}^{\textnormal{inv}} c^{v , \alpha , \gamma})_{\bullet}$, and so $({}^{\textnormal{inv}} c^{v,\alpha , \gamma})_0 = \alpha (v_2)$. It is therefore enough to show that, for a different $\delta \in T^{++}$, the decomposition ${}^K V = W_1 \oplus W_2$ is the same. We can find $n \in \bbZ_{\ge 1}$ such that $\delta^n / \gamma \in T^{++}$ and then it is clear that $e_K \delta e_K$ is nilpotent on $W_1$. Similarly, we can find $n \in \bbZ_{\ge 1}$ such that $\gamma^n / \delta \in T^{++}$ and then it is clear that $e_K \delta e_K$ is invertible on $W_2$. From this the desired follows.

\subsubsection{}

It is immediate to check the functoriality of the above construction, giving rise to the desired  morphism of functors $$ \widetilde{b}_{H , \bbP , \bbL} : Forg|_{\calM (G)^{\fl}} \to Coinv_H|_{\calM(G)^{\fl}}.$$

\subsection{The morphism $b_{H,\bbP , \bbL}$}

\subsubsection{}

We claim that the morphism $$ \widetilde{b}_{H,\bbP , \bbL} : Forg|_{\calM (G)^{\fl}} \to Coinv_H |_{\calM (G)^{\fl}}$$ factors via $$ Forg|_{\calM (G)^{\fl}} \to Coinv_{H_{\bbP, \bbL}} |_{\calM (G)^{\fl}},$$ and we then denote by $$ b_{H,\bbP , \bbL} : Coinv_{H_{\bbP , \bbL}} |_{\calM (G)^{\fl}} \to Coinv_H |_{\calM (G)^{\fl}}$$ the resulting morphism.

\subsubsection{}

Indeed, given $V \in \calM (G)^{\fl}$, given $v \in V$ and $\alpha \in V^{*,H}$ and given $g \in H_{\bbP, \bbL}$, it is enough to check that the sequences $n \mapsto \alpha (\gamma^n v)$ and $n \mapsto \alpha (\gamma^n g v)$ are eventually equal. This follows directly from the definition of $H_{\bbP, \bbL}$.

\subsection{Algebraicity of the morphism $b_{H,\bbP , \bbL}$}

\subsubsection{}

In this subsection we prove:

\begin{theorem}\label{thm b is algebraic}
	The morphism $$ b_{H,\bbP , \bbL} : Coinv_{H_{\bbP , \bbL}} |_{\calM (G)^{\fl}} \to Coinv_H |_{\calM (G)^{\fl}}$$ is algebraic (in the sense of \S\ref{sssec algebraic morphism}).
\end{theorem}

It is immediate that $b_{H,\bbP , \bbL}$ is algebraic if and only if $\widetilde{b}_{H,\bbP , \bbL}$ is algebraic.

\subsubsection{}\label{sssec enough admissibles needed}

By Corollary \ref{cor cat pal wie for Coinv}, to show that $\widetilde{b}_{H,\bbP , \bbL}$ is algebraic it is enough to check that the data $(\widetilde{b}_{H,\bbP , \bbL} , Z , M , m , \alpha)$ is generically algebraic if $Z \cong \bbA^n$ for $n \in  \bbZ_{\ge 1}$ and in fact algebraic if $Z \cong \bbA^1$. Here $m \in M$ and $\alpha \in \Hom_{\calO_Z} ({}_H M , \calO_Z)$.

\subsubsection{}

Let us fix an open compact subgroup $K \subset G$ exactly as in \S\ref{sssec choosing L and K} (substituting $M$ for $V$and $m$ for $v$). Let us denote by $\tau$ the $\calO_Z$-linear endomorphism of ${}^K M$ given by acting by $e_K \gamma e_K$.

\subsubsection{}\label{sssec exists f with unit}

Suppose that that there exists a monic polynomial $f \in \calO_Z [x]$ such that $f(\tau) = 0$, and whose first non-zero coefficient is a unit in $\calO_Z$. Let $a \in \calO_Z^{\times}$ be the first non-zero coefficient of $f$. Write $f(x) = x^d \cdot (x f^{\prime} (x) + a)$ where $d \in \bbZ_{\ge 0}$ and $f^{\prime} \in \calO_Z [x]$. Then, since $a$ is a unit in $\calO_Z$, the polynomials $x^d$ and $x f^{\prime} (x) + a$ are relatively prime, and therefore we can find $f_1 , f_2 \in \calO_Z [x]$ such that $f_1 (x) x^d + f_2 (x) (x f^{\prime} (x) + a) = 1$. Set $h(x) := f_1 (x) x^d$ and define $c := \alpha (h(\tau) m) \in \calO_Z$.

\medskip

We first claim that if $Z$ is $0$-dimensional, then $c$ coincides with the image of $m$ under $M \xrightarrow{(\widetilde{b}_{H,\bbP,\bbL})_M} {}_H M \xrightarrow{\alpha} \calO_Z$. Indeed, notice that $h(\tau)m$ lies in $\Ker (\tau \cdot f^{\prime} (\tau) + a)$, which is a $\tau$-invariant subspace of ${}^K M$ on which $\tau$ is invertible. The difference $m - h(\tau)m$ lies in $\Ker (\tau^d)$, which is a $\tau$-invariant subspace of ${}^K M$ on which $\tau$ is nilpotent. Therefore, following \S\ref{sssec second description of inv}, we see that, given  any functional $\ell \in \calO_W^*$, we have $\ell (\alpha (\widetilde{b}_{H,\bbP , \bbL} (m))) = \ell (\alpha (h(\tau)m)) = \ell (c)$. Since $\ell$ was arbitrary, we get $\alpha (\widetilde{b}_{H,\bbP , \bbL} (m)) = c$, as desired.

\medskip

Not assuming that $Z$ is $0$-dimensional anymore, we claim that $c$ attests to $(\widetilde{b}_{H,\bbP , \bbL} , Z , M , m , \alpha)$ being algebraic. In other words, that for every $0$-dimensional closed subscheme $W \subset Z$, the image of $c$ under $\calO_Z \to \calO_W$ coincides with the image of $m$ under $$ M \to \calO_W \underset{\calO_Z}{\otimes} M \xrightarrow{(\widetilde{b}_{H,\bbP,\bbL})_{\calO_W \underset{\calO_Z}{\otimes} M}} {}_H (\calO_W \underset{\calO_Z}{\otimes} M) \cong \calO_W \underset{\calO_Z}{\otimes} {}_H M \xrightarrow{\alpha} \calO_W.$$ This is clear from the preceding discussion of the $0$-dimensional case and $\calO_W \underset{\calO_Z}{\otimes} {}^K M \cong {}^K (\calO_W \underset{\calO_Z}{\otimes} M)$.

\subsubsection{}

Building on \S\ref{sssec exists f with unit}, we can now see that $(\widetilde{b}_{H,\bbP , \bbL} , Z , M , m , \alpha)$ is always generically algebraic, say assuming that $Z$ is integral (as we anyway only need to know this for $Z$ being an affine space). Namely, choose a monic polynomial $f \in \calO_Z [x]$ such that $f (\tau) = 0$, and let $a \in \calO_Z$ be the first non-zero coefficient of $f$. Then we can pullback to $Z_a$, obtaining data $(\widetilde{b}_{H,\bbP , \bbL} , Z_a , \calO_{Z_a} \underset{\calO_Z}{\otimes} M , [m] , [\alpha])$, for which we can take the same polynomial $f \in \calO_Z [x] \to \calO_{Z_a} [x]$, whose first non-zero coefficient in $\calO_{Z_a}$ is already a unit.

\subsubsection{}

It is left to show that if $Z \cong \bbA^1$ then $(\widetilde{b}_{H,\bbP , \bbL} , Z , M , m , \alpha)$ is algebraic, by showing that we can find in this case $f \in \calO_Z [x]$ and $n_0 \in \bbZ_{\ge 1}$ such that $f (\tau^{n_0}) = 0$ and such that the first non-zero coefficient of $f$ is a unit in $\calO_Z$ (notice that we can replace $\gamma$ with $\gamma^{n_0}$, and so $\tau$ with $\tau^{n_0}$). By the stabilization theory of Bernstein (\cite[\S 3.3, Lemma 34]{BeRu} coupled with \cite[\S 3.2, Jacquet's lemma (final version)]{BeRu}), for some $n_0 \in \bbZ_{\ge 1}$ the endomorphism $\tau^{n_0}$ of ${}^K V$ is stable, i.e. ${}^K V = \Ker (\tau^{n_0}) \oplus \Image (\tau^{n_0})$ and the restriction of $\tau^{n_0}$ to $\Image (\tau^{n_0})$ is invertible. We will be therefore done using the following lemma:

\begin{lemma}\label{lem stable is good}
	Let $A = \bbC [y]$ be the ring of polynomials in one variable. Let $N$ be a finitely generated $A$-module. Let $\tau : N \to N$ be an $A$-linear endomorphism. Assume that $\tau$ is stable, i.e. $N = \Ker (\tau) \oplus \Image (\tau)$ and $\tau |_{\Image (\tau)} : \Image (\tau) \to \Image (\tau)$ is invertible. Then there exists a monic polynomial $f \in A [x]$, whose first non-zero coefficient is a unit in $A$, such that $f (\tau) = 0$.
\end{lemma}

\begin{proof}
	Given in \S\ref{ssec proof of lem stable is good}.
\end{proof}

\subsection{The Bernstein morphism $B_{H,\bbP, \bbL}$}

\subsubsection{}

Since $$ b_{H,\bbP, \bbL} : Coinv_{H_{\bbP, \bbL}} |_{\calM (G)^{\fl}} \to Coinv_H |_{\calM (G)^{\fl}}$$ is algebraic, using the previous results of this article we obtain that it extends uniquely to a morphism $$ B_{H,\bbP, \bbL} : Coinv_{H_{\bbP, \bbL}} \to Coinv_H.$$ Since for $D,C \in \Fun^l (\calM (G) , \Vect)$ we have (by restriction) $$ \Hom (D,C) \cong \Hom_{\calH_G^{\op}} (D(\calH_G)  , C(\calH_G)),$$ we obtain $$ \Hom (Coinv_{H_{\bbP, \bbL}} , Coinv_{H}) \cong \Hom_{\calH_G^{\op}} ( Coinv_{H_{\bbP, \bbL}} (\calH_G) , Coinv_H (\calH_G) ) \cong $$ $$ \cong \Hom_{G^{\op}} (\calD_c (H_{\bbP, \bbL} \bs G)^{\sm} , \calD_c (H \backslash G)^{\sm}).$$ Therefore the information of our morphism $B_{H,\bbP, \bbL}$ is equivalent to that of a $G$-equivariant morphism (which we denote by abuse of notation by the same name) $$ B_{H , \bbP, \bbL}, :  \calD_c (H_{\bbP, \bbL} \bs G)^{\sm} \to \calD_c (H \bs G)^{\sm},$$ which is called the \textbf{Bernstein morphism}.

\subsection{Another interpretation}

\subsubsection{}

Recall isomorphism (\ref{eq isom fnctnl fnctr}). Applying it to $\widetilde{b}_{H,\bbP,\bbL}$, we obtain a functional $\textnormal{asymp}_{H , \bbP,\bbL} \in \calA(H \backslash G)^*$. Thus, $\textnormal{asymp}_{H , \bbP , \bbL} (\phi)$ is equal to the value at $0$ of the ${}^{\textnormal{inv}} (-)$-part of the sequence $n \mapsto \phi (H \gamma^n)$.

\subsubsection{}

It is clear that the functional $\textnormal{asymp}_{H , \bbP,\bbL}$ is $H_{\bbP , \bbL}$-invariant.

\subsubsection{}

Applying Frobenius reciprocity to the $H_{\bbP , \bbL}$-invariant functional $\textnormal{asymp}_{H , \bbP,\bbL}$, we see that there exists a unique $G$-equivariant linear map $$ \textnormal{Asymp}_{H , \bbP,\bbL} : \calA (H \bs G) \to \calA (H_{\bbP,\bbL} \bs G)$$ which, when composed with the evaluation at the origin of $H_{\bbP , \bbL} \bs G$, yields $\textnormal{asymp}_{H , \bbP,\bbL}$.

\subsubsection{}

We can now re-interpret the Bernstein morphism construction as follows:

\begin{theorem}\label{thm Asymp is continuous}
	The map $$ \textnormal{Asymp}_{H , \bbP,\bbL} : \calA (H \bs G) \to \calA (H_{\bbP,\bbL} \bs G)$$ is continuous (with repsect to the topologies  of \S\ref{ssec function spaces}).
\end{theorem}

\begin{proof}

	We need to show that pre-composing a continuous functional on $\calA (H_{\bbP , \bbL} \bs G)$ with $\textnormal{Asymp}_{H,\bbP,\bbL}$ yields a continuous functional on $\calA (H \bs G)$. Each continuous functional on $\calA (H_{\bbP , \bbL} \bs G)$ is the restriction of a continuous functional on ${}^{\sm} C^{\infty} (H_{\bbP , \bbL} \bs G)$, the space of which we denoted by $\calD_{\pc} (H_{\bbP , \bbL} \bs G)$. So we need to show that the image of the map $ \calD_{\pc} (H_{\bbP , \bbL} \bs G) \to \calA (H \bs G)^*$ dual to $\textnormal{Asymp}_{H,\bbP,\bbL}$ consists of continuous functionals.
	We have the following commutative diagram: $$ \xymatrix{ \Hom (Forg , Coinv_{H_{\bbP,\bbL}}) \ar@{<->}[d]_{\sim} \ar[r] & \Hom (Forg|_{\calM (G)^{\fl}} , Coinv_H|_{\calM (G)^{\fl}} ) \ar@{<->}[d]^{\sim} \\ \calD_{\pc} (H_{\bbP,\bbL} \bs G) \ar[r] & \calA (H \bs G)^* }.$$ Here the lower map is our dual to $\textnormal{Asymp}_{H,\bbP,\bbL}$. The upper map is composition with $b_{H,\bbP,\bbL}$ (preceded with restriction to $\calM (G)^{\fl}$). The left isomorphism is isomorphism (\ref{eq functor hom description}) and the right isomorphism is isomorphism (\ref{eq isom fnctnl fnctr}). The commutativity of the diagram is a straight-forward chasing of the definitions. Since have established that continuous functionals on the lower right are the same as algebraic functionals, and algebraic functionals match to algebraic morphisms on the top right, we see that it is enough to see that the image of the upper arrow consists of algebraic morphisms. But this is clear since we have already established that $b_{H , \bbP , \bbL}$ is an algebraic morphism and so extends to a morphism $B_{H , \bbP , \bbL} : Coinv_{H_{\bbP,\bbL}} \to Coinv_H$, and so our upper arrow is equal to composition with $B_{H , \bbP , \bbL}$ followed by restriction to $\calM (G)^{\fl}$, so clearly its image consists of algebraic morphisms.
\end{proof}

\subsubsection{}

To conclude, in this interpretation the Bernstein map $$ B_{H , \bbP, \bbL}, :  \calD_{\pc} (H_{\bbP, \bbL} \bs G) \to \calD_{\pc} (H \bs G)$$ is the dual of $\textnormal{Asymp}_{H , \bbP,\bbL}$, well defined since the latter is continuous.

%%%%%%%%%%%%%%%%%%%%%%%%%%%%%%%%%%%%
%%%%%%%%%%%%%%%%%%%%%%%%%%%%%%%%%%%%

\section{Proofs}\label{sec remaining proofs}

In this section we provide proofs for the statements which were left without proof earlier in the article.

\subsection{Proof of Claim \ref{clm criterion rpfl}}\label{ssec proof of clm criterion rpfl}

\subsubsection{}

Denote $R := \End (P)^{\op}$. The essential image of $$P \underset{R}{\otimes} - : \calM (R) \to \calM$$ is a direct summand in $\calM$, in particular a Serre subcategory. Therefore, it is clear that we can replace $(\calM,P)$ of the claim with $(\calM(R) , R)$.

\subsubsection{}

First we show that $R$ is punctually of finite length as an $A$-module (recall that $A$ stands for the center of $R$). Let $I \subset A$ be a cofinite ideal. Since $R$ is a finitely-generated $A$-module, $R / IR$ is finite-dimensional, and therefore of finite length.

\subsubsection{}

Next, for the ``regularly" part, we need to show that given a morphism of $R$-modules $R \to V$, where $V$ is an $R$-module of finite length, there exists a cofinite ideal $I \subset A$ such that $R \to V$ factors via $R \to R/RI$. An easy inductive argument shows that it is enough to treat the case when $V$ is a simple $R$-module and $R \to V$ is a surjection. Then $R \to V$ is isomorphic to the canonical projection $R \to R / \frakm$ for a left maximal ideal $\frakm \subset R$. Denote $\frakn := \frakm \cap A$. Then $\frakn$ is an ideal in $A$ (not equal to $A$ itself). Let $J \subset A$ be an ideal containing $\frakn$ properly. Then $RJ + \frakm$ is a left ideal in $R$ containing $\frakm$ properly, and hence $R = RJ + \frakm$. So $J(R/\frakm) = R / \frakm$ and therefore by Nakayama's lemma there exists $j \in J$ such that $(1+j)R \subset \frakm$, so $1 + j \in \frakm \cap A = \frakn$, so $1 \in J$, i.e. $J = A$. Therefore $\frakn$ is a maximal ideal in $A$, and in particular a cofinite ideal. As $R \to R / \frakm$ factors via $R / R \frakn$, we are done.

\subsection{Proof of Claim \ref{clm criterion finite generation}}\label{ssec proof of clm criterion finite generation}

\subsubsection{}

Let $P \in \calM$ be a compact projective object. Notice that the direct sum of a finite family of compact projective objects $P$ for which $C(P)$ is a finitely generated $\End(P)$-module again has this property. Therefore we can find a compact projective object $Q \in \calM$ for which $C(Q)$ is a finitely generated $\End(Q)$-module, such that $P$ is a direct summand of $Q$. Write $Q = P \oplus P^{\prime}$. Then since $C(Q) \cong C(P) \oplus C(P^{\prime})$ is finitely generated over $\End(Q)$, we can find finitely many elements $c_1 , \ldots , c_a \in C(P)$ and $c_1^{\prime} , \ldots , c_b^{\prime} \in C(P^{\prime})$ such that $$C(Q) = \sum \End (Q) \cdot c_i + \sum \End (Q) \cdot c_j^{\prime},$$ and decomposing we see in particular that $$ C(P) = \sum \End (P) \cdot c_i + \sum \Hom (P^{\prime} , P) \cdot c_j^{\prime}.$$ Therefore, it is enough to show that $\Hom (P^{\prime} , P)$ is a finitely generated $\End(P)$-module.

\subsubsection{}

So, we are left to show that for two compact projective objects $P,P^{\prime} \in \calM$, $\Hom (P^{\prime} , P)$ is a finitely generated $\End (P)$-module. The module $P^{\prime}$ can be realized as a direct summand in a finite direct sum of splitting familial compact projective objects, and so we can replace it by this direct sum, and then replace this direct sum by a single summand. In other words, we can reduce to assuming that $P^{\prime}$ itself is a splitting familial compact projective object. Then we can replace $P$ by its component in the direct summand of $\calM$ which is the essential image of $- \underset{\End (P^{\prime})}{\otimes} P^{\prime}$ (as $P^{\prime}$ is splitting). Now, denoting $R := \End(P^{\prime})^{\op}$, we can, via the equivalence of $\calM(R)$ with this direct summand, translate the problem to $\calM (R)$. Namely, we are given a compact projective object $S \in \calM (R)$, and want to show that $S$ is finitely generated over $\End_R (S)$. That $S$ is a compact projective in $\calM (R)$ simply means that $S$ is a finitely generated projective $R$-module. Now, as $P^{\prime}$ is familial, $R$ is a finitely generated module over its center $A$. Then $S$ is a finitely generated module over $A$. Since $A$ acts on $S$ via $A \to \End_R (S)$, we obtain that $S$ is indeed finitely generated over $\End_R (S)$.

\subsection{Proof of Proposition \ref{prop abstract setting}}\label{ssec proof of prop abstract setting}

\subsubsection{}

To abbreviate, by a \textbf{good pair} $(A,M)$ we will mean a pair in $\pi$, such that additionally $M \in \calM^A$ is regularly punctually of finite length, and such that $C(M)$ is a finitely generated $A$-module.

\subsubsection{}\label{sssec good is injective}

Let $(A,M)$ be a good pair. Let us notice first that since $C(M)$ is a finitely generated $A$-module, as we run over morphisms $\underline{M} \to V$ where $V \in \calM^{\fl}$, the resulting maps $C(\underline{M}) \to C(V)$ are jointly injective. Indeed, since those maps include the projection maps $$C(M) \to C(M/IM) \cong C(M) / I C(M)$$ as $I$ runs over the set of cofinite ideals in $A$, this follows from $C(M)$ being a finitely generated $A$-module and Krulls intersection theorem.

\subsubsection{}\label{sssec pair is good}

Let $(A,M)$ be a good pair. We will show that there exists a unique linear map $$ T_{\underline{M}} : D(\underline{M}) \to C(\underline{M})$$ such that for every morphism $\underline{M} \to V$ with $V\in \calM^{\fl}$, the following diagram commutes: $$ \xymatrix{D(\underline{M}) \ar[r]^{T_{\underline{M}}} \ar[d] & C(\underline{M}) \ar[d] \\ D(V) \ar[r]^{t_V} & C(V)}.$$ Uniqueness follows from \S\ref{sssec good is injective}.

\subsubsection{}

Assuming what is claimed in \S\ref{sssec pair is good}, let us now finish the proof of Proposition \ref{prop abstract setting}. First, given good pairs $(A,M), (B,N)$, by \S\ref{sssec good is injective}, it is immediate to see that for every morphism $\underline{M} \to \underline{N}$, the diagram $$ \xymatrix{ D(\underline{M}) \ar[r]^{T_{\underline{M}}} \ar[d] & C(\underline{M}) \ar[d] \\ D(\underline{N}) \ar[r]^{T_{\underline{N}}} & C(\underline{N})} $$ commutes. Denote by $\calP \subset \calM$ the full subcategory consisting of objects $\underline{M}$ for a good pair $(A,M)$, which are compact projective. Aggregating all the $T_{\underline{M}}$'s, we obtain a morphism of functors $$ T : D|_{\calP} \to C|_{\calP}$$ extending the morphism $$ t : D|_{\calM^{\fl}} \to C|_{\calM^{\fl}}.$$ By the assumption of Proposition \ref{prop abstract setting}, $\calP$ generates the category $\calM$, and therefore the restriction $$ \Hom (D , C) \to \Hom (D|_{\calP} , C|_{\calP})$$ is an isomorphism. Lifting $T$ under this isomorphism, the proof of Proposition \ref{prop abstract setting} is done (the injectivity of the map of $\Hom$-spaces in the formulation of Proposition \ref{prop abstract setting} is also now clear, in view of the uniqueness of the $T_{\underline{M}}$'s).

\subsubsection{}

It is left to show what was promised in \S\ref{sssec pair is good}. Let $(A,M)$ be a good pair. We will proceed by induction on the dimension of $Spec(A)$. If the dimension of $Spec(A)$ is equal to $0$, the claim is clear since in that case $\underline{M}$ has finite length and we can take $T_{\underline{M}}$ to be $t_{\underline{M}}$. So we assume that the dimension of $Spec(A)$ if bigger than $0$, and that the claim was showed for smaller dimensions. Furthermore, we can assume that $Spec(A)$ is an affine space. Indeed, by Noether's normalization lemma, we can find a finite morphism $B \to A$ in $\Alg$ such that $Spec(B)$ is an affine space. Then $(B, M^A_B)$ is a good pair and $\underline{M^A_B} = \underline{M}$, so that we can replace $A$ with $B$.

\medskip

We want to show that for every $d \in D(M)$ there exists $c \in C(M)$ such that for every morphism $\underline{M} \to V$ where $V \in \calM^{\fl}$, the image of $c$ under $C(\underline{M}) \to C(V)$ coincides with the image of $d$ under $D(\underline{M}) \to D(V) \xrightarrow{t_V} C(V)$. Here, since $M$ is regularly punctually of finite length, it is enough to restrict attention to $\underline{M} \to V$ running over the canonical $\underline{M} \to \underline{M/IM}$ as $I$ runs over cofinite ideals in $A$. Let us denote by $\calI_A$ the set of ideals $I \subset A$ for which the dimension of $Spec(A/I)$ is smaller than the dimension of $Spec(A)$ (which, since $Spec(A)$ is smooth and connected, are the same as non-zero ideals in $A$). For every $I \in \calI_A$, let us denote by $c_I$ the image of $d$ under $$D(M) \to D(M/IM) \xrightarrow{T_{M/IM}} C(M/IM) \cong C(M) / I C(M)$$ (this is possible by the induction hypothesis, since the dimension of $Spec(A/I)$ is smaller than that that of $Spec(A)$ and hence $(A/I , M/IM)$ is a good pair for which the claim is already known). We thus obtain a system of elements $$ (c_I)_{I \in \calI_A} \in \prod_{I \in \calI_A} C(M)/IC(M),$$ coordinated in the sense that given $J \subset I$, the image of $c_J$ under the projection $C(M)/J C(M) \to C(M) / I C(M)$ is equal to $c_I$. We would like to see that there exists an element $c \in C(M)$ such that for every $I \in \calI_A$, the image of $c$ under the projection $C(M) \to C(M) / I C(M)$ is equal to $c_I$. What we also know about our system $(c_I)_{I \in \calI_A}$, by the $\pi$-$\bbA$-almost algebraicity of $t$, is as follows. Let $\alpha : C(M) \to A$ be a morphism of $A$-modules. Then the resulting system $(\alpha(c_I))_{I \in \calI_A} \in \prod_{I \in \calI_A} A/I$ does come from an element of $A$ in case that the dimension of $Spec(A)$ is $1$, and generically comes from an element of $A$ in case that the dimension of $Spec(A)$ is strictly larger than $1$. Proposition \ref{prop app main} and Proposition \ref{prop app main 2} then guarantee that there exists an element $c \in C(M)$ as desired.

\subsection{Proof of Claim \ref{clm M(G) and Coinv_H are suitable}}\label{ssec proof of clm M(G) and Coinv_H are suitable}

\subsubsection{}

The first condition follows from Bernstein's theory of decomposition of $\calM (G)$ with respect to the center (see, for example, \cite[\S VI]{Re} and \cite{BeRu}). Namely, let $\mathbb{P} \subset \mathbb{G}$ be a parabolic with Levi quotient $\mathbb{P} \to \mathbb{L}$ (and denote as usual $L := \mathbb{L} (F)$). Then considering the algebra $\calO_{Z_L}$ of regular functions on the variety $Z_L$ of unramified characters of $L$, and a cuspidal irreducible module $E \in \calM (L)$, we have the normalized parabolic induction $$P_{L,E} := pind^L_G (E \otimes \calO_{Z_L}) \in \calM (G).$$ Then it is known that $P_{L,E}$ is a splitting compact projective object in $\calM (G)$ (the subcategory it generates is the Bernstein component corresponding to the cuspidal data $(L,E)$). As we run over pairs $(L,E)$, the $P_{L,E}$ generate the category $\calM (G)$. Each $P_{L,E}$ is familial, and for every open compact subgroup $K \subset G$, ${}^K P_{L,E}$  is finitely generated over the cetner of $\End (P_{L,E})$.

\subsubsection{}

We now show that the second condition holds. There are enough compact projective objects in $\calM (G)$ of the form $\calH_G e_K$ where $e_K \in \calH_G$ is the idempotent corresponding to an open compact subgroup $K \subset G$. Then $\End (\calH_G e_K)$ is the opposite of the Hecke algebra of $K$-biinvariant compactly supported distributions on $G$, while $Coinv_H (\calH_G e_K)$ is the space of right $K$-invariants in $\calD_c (H \backslash G)$. Then \cite[Theorem I]{AiGoSa} states that $Coinv_H (\calH_G e_K)$ is finitely generated over $\End (\calH_G e_K)$.

\subsection{Proof of Lemma \ref{lem stable is good}}\label{ssec proof of lem stable is good}

\subsubsection{}

We can assume that $\tau$ is invertible, since if $f \in A [x]$ satisfies $f (\tau|_{\Image (\tau)}) = 0$ and the first non-zero coefficient of $f$ is a unit in $A$, then $f^{\prime} (x) := x \cdot f(x)$ will satisfy $f^{\prime} (\tau) = 0$ and the first non-zero coefficient of $f^{\prime}$ is a unit in $A$.

\subsubsection{}

In our situation, let us say that $(N,\tau)$ is good if we can find a list of generators $v_1 , \ldots , v_d$ of $N$ as a $\bbC [y]$-module, and elements $c_{i,j} \in \calO_Z$ for which $\tau (v_i) = \sum_{j} c_{i,j} \cdot v_j$, such that the determinant of the matrix $(c_{i,j})$ is a unit in $\bbC [y]$. If $(N,\tau)$ is good then we are done, since, considering the characteristic polynomial $f \in (\bbC [y]) [x]$ of the matrix $(c_{i,j})$, we have that $f(\tau) = 0$ and the zeroth coefficient of $f$, plus minus the determinant of $(c_{i,j})$, is a unit in $\bbC [y]$.

\subsubsection{}

Suppose that $L \subset N$ is an $\bbC [y]$-submodule such that $\tau (L) = L$. Then $\tau$ induces invertible endomorphisms $\tau_L$ and $\tau_{N/L}$ on $L$ and on $N/L$. We claim that if $(L,\tau_L)$ and $(N/L , \tau_{N/L})$ are good, then $(L,\tau)$ is good. Indeed, we choose generators and coefficients as above of $L$, generators and coefficients as above of $N/L$, and lift the latter to elements of $N$, completing the lists of coefficients. The combined list of generators and coefficients will clearly be suitable, as the resulting matrix will be block triangular, so with determinant the product of determinants assumed to be units in $\bbC [y]$.

\subsubsection{}

We can consider the submodule $N_{\textnormal{trs}} \subset N$ consisting of torsion elements and, since $\bbC [y]$ is a principal ideal domain, the quotient $N / N_{\textnormal{trs}}$ will be free. Clearly $\tau (N_{\textnormal{trs}}) = N_{\textnormal{trs}}$. Since the goodness in the free case is clear, we are reduced to assuming that $N$ is torsion.

\subsubsection{}

Assuming that $N$ is torsion, and non-zero, there exists $c \in \bbC$ such that the submodule $N^{y-c} \subset N$ consisting of elements annihilated by $y-c$ is non-zero. Clearly $\tau (N^{y-c}) = N^{y-c}$. Using induction, we can assume that $(N / N^{y-c} , \tau_{N / N^{y-c}})$ is good, and so it is enough to prove that $(N^{y-c},\tau_{N^{y-c}})$ is good. However, $N^{y-c}$ is a finite-dimensional $\bbC$-module, and the claim is clear by choosing a $\bbC$-basis of it and taking the coefficients $c_{i,j}$ as above to be in $\bbC$.

%%%%%%%%%%%%%%%%%%%%%%%%%%%%%%%%%%%%
%%%%%%%%%%%%%%%%%%%%%%%%%%%%%%%%%%%%

\section{Propositions on gluing sections on closed subschemes}\label{sec algebraic lemma}

The purpose of this section is to prove Proposition \ref{prop app main} and Proposition \ref{prop app main 2}, which will be used for the induction on dimension in the proof of Proposition \ref{prop abstract setting}. Throughout this section, we fix an irreducible\footnote{One can probably change statements so as to drop the irreducibility assumption (as well as a normality assumption later on), however this will slightly complicate things, and we willl anyhow only apply the propositions to smooth connected $Z$.} $Z \in \Aff$. We will write $\calM (Z)$ for $\calM (\calO_Z)$. Given locally closed $W \subset Z$ and $M \in \calM (Z)$, we will write $M|_W$ for $\calO_W \underset{\calO_Z}{\otimes} M$.

\medskip

For a $B$-module $M$ and $f \in B$, we use the following notation for the $f$-torsion: $$ M^f := \{ m \in M \ | \ f^n m = 0 \textnormal{ for some } n \in \bbZ_{\ge 1}\},$$ and the following notation for the $f$-completion: $$ M^{\wedge}_f := \lim_{n \in \bbZ_{\ge 1}} M/f^n M.$$

\subsection{Formulation of the propositions}

\subsubsection{}

Denote by $\calI_Z$ the partially ordered set of closed subschemes $W \subset Z$ which are either $0$-dimensional or of dimension strictly smaller than that of $Z$. We define a functor $$ S: \calM (Z) \to \calM (Z)$$ by setting $$ S(M) := \lim_{W \in \calI_Z} M|_W.$$ For $s \in S(M)$, we will denote by $s_W$ the component of $s$ in $M|_W$.

\subsubsection{}

We have a canonical morphism $Id \to S$.

\begin{lemma}
    Let $M \in \calM (Z)$ be finitely generated. Then $M \to S (M)$ is injective.
\end{lemma}

\begin{proof}
    This readily follows from Krull's intersection theorem.
\end{proof}

\subsubsection{}

Let $M \in \calM (Z)$ be finitely generated. We say that an element in $S(M)$ is \textbf{coherent} if it lies in the image of $M \to S(M)$. We say that an element $s \in S(M)$ is \textbf{generically coherent} if there exists a dense distinguished open subset $U \subset Z$ and $m \in M|_U$ such that for every $0$-dimensional closed subscheme $W \subset U$, the image of $m$ under $M|_U \to M|_W$ is equal to $s|_W$.

\subsubsection{}

The following propositions are the main results of this section:

\begin{proposition}\label{prop app main}
    Let $M \in \calM (Z)$ be finitely generated and let $s \in S(M)$. Suppose that, for every morphism of $\calO_Z$-modules $\alpha : M \to \calO_Z$, we have that $\alpha (s) \in S(\calO_Z)$ is coherent. Then $s$ is coherent.
\end{proposition}

\begin{proposition}\label{prop app main 2}
    Assume that $Z$ is normal and that the dimension of $Z$ is strictly larger than $1$. Let $M \in \calM (Z)$ be finitely generated and let $s \in S(M)$. Suppose that for every morphism of $\calO_Z$-modules $\alpha : M \to \calO_Z$, we have that $\alpha (s) \in S(\calO_Z)$ is generically coherent. Then $s$ is coherent.
\end{proposition}

\subsection{Proof of Proposition \ref{prop app main}}

\subsubsection{}\label{sssec app torsion case}

We first prove the following lemma, which will also be used in the proof of Proposition \ref{prop app main 2}:

\begin{lemma}\label{lem torsion coherent}
	Let $M \in \calM (Z)$ be finitely generated and let $s \in S(M)$. Assume that there exists a dense distinguished open subset $U \subset Z$ such that, for every $W \in \calI_Z$, $(s_W)|_{W \cap U} = 0$. Then $s$ is coherent.
\end{lemma}

\begin{proof}

Write $U = Z_f$ for $f \in \calO_Z$. Let us consider the ``forgetting" map $$ S(M) \to M^{\wedge}_f$$ and the image $s^{\wedge}$ of $s$ under this map. For every $0$-dimensional $W \in \calI_Z$, the image of $s^{\wedge}$ under $$ M^{\wedge}_f \to (M^{\wedge}_f) |_W \cong (M|_W)^{\wedge}_f$$ coincides with the image of $s_W$ under $M|_W \to (M|_W)^{\wedge}_f$. Hence, this image is $f$-torsion. By Lemma \ref{lem app lemma 3} that we will establish, $s^{\wedge}$ is therefore itself $f$-torsion. By Lemma \ref{lem app lemma 4} that we will establish, there exists therefore a unique $m \in M^f$ which maps to $s^{\wedge}$ under $M \to M^{\wedge}_f$. We want to show finally that $s^{\prime} := s - [m] \in S (M)$ is $0$, where $[m]$ is the image of $m$ under $M \to S(M)$. The properties of $s^{\prime}$ are that $s^{\prime}_W \in (M|_W)^f$ for every $W \in \calI_Z$, and that $s^{\prime}_{f^n \calO_Z} = 0$ for every $n \in \bbZ_{\ge 1}$. Therefore, for a given $W \in \calI_Z$ the element $s^{\prime}_W \in (M|_W)^f$ maps to zero under $M|_W \to (M|_W)^{\wedge}_f$. By Lemma \ref{lem app lemma 4} that we will establish, applied to $M|_W$, we see that $s^{\prime}_W = 0$.
\end{proof}

\subsubsection{}\label{sssec app reflexive case}

Let us now prove Proposition \ref{prop app main} assuming that $M$ is reflexive, i.e. the morphism $$ M \to R(M) := \Hom_{\calO_Z} (\Hom_{\calO_Z} (M , \calO_Z),\calO_Z)$$ is an isomorphism. Let $\alpha : M \to \calO_Z$. Since by assumption $\alpha (s)$ is coherent, we obtain a corresponding element $a_{\alpha} \in \calO_Z$. The assignment $\alpha \mapsto a_{\alpha}$ is readily seen to define an element of $R(M)$, and so an element $m \in M$. Let us denote $s^{\prime} := s - [m]$ (where $[m]$ is the image of $m$ under $M \to S(M)$) - it is enough now to see that $s^{\prime}$ is coherent. We have that $\alpha (s^{\prime}) = 0$ for any $\alpha : M \to \calO_Z$. Let $U \subset Z$ be a dense distinguished open subset such that $M|_U$ is a locally free $\calO_U$-module. Let $W \in \calI_Z$. Then the image of $s^{\prime}_W$ under $M|_W \to M|_{W \cap U}$ maps to $0$ under $M|_{W \cap U} \to \calO_{W \cap U}$ for any $\calO_Z$-module morphism $\alpha : M \to \calO_Z$. Since $M|_U$ is a locally free $\calO_U$-module we obtain that this image itself is $0$, i.e. $(s^{\prime}_W)|_{W \cap U} = 0$, for all $W \in \calI_Z$. By Lemma \ref{lem torsion coherent} we see that $s^{\prime}$ is coherent.

\subsubsection{}

Let now $M$ be general again. We consider the cokernel sequence $$ M \to R(M) \to C \to 0.$$ Consider the commutative diagram $$ \xymatrix{M \ar[r] \ar[d] & R(M) \ar[r] \ar[d] & C \ar[r] \ar[d] & 0 \\ S (M) \ar[r] & S (R(M)) \ar[r] & S (C)}.$$ Since $R(M)$ is reflexive and therefore the claim has already been established for it in \S\ref{sssec app reflexive case}, a simple diagram chasing shows that there exists $m \in M$ such that $s^{\prime} := s - [m]$ maps to zero under $S (M) \to S (R(M))$ (where $[m]$ is the image of $m$ under $M \to S(M)$) - and it is enough now to check that $s^{\prime}$ is coherent. Let $U \subset Z$ be a dense distinguished open subset such that $M|_U$ is a locally free $\calO_U$-module. Then $M|_U \to R(M)|_U$ is an isomorphism. Therefore we see that, for any $W \in \calI_Z$, $(s^{\prime}_W)|_{W \cap U} = 0$. Hence, again invoking Lemma \ref{lem torsion coherent}, we obtain that $s^{\prime}$ is coherent.

\subsection{Auxiliary lemmas used in the proof of Lemma \ref{lem torsion coherent}}

Throughout this subsection, we work with algebras instead of affine schemes, so we let $A \in \Alg$, and by $\calI_A$ we denote the partially ordered set of ideals $I \subset A$ for which $Spec(A/I)$ is either $0$-dimensional or of dimension strictly less than that of $Spec(A)$.

\subsubsection{} The first auxiliary lemma is Lemma \ref{lem app lemma 3}, to which we arrive after two preparatory lemmas.

\begin{lemma}\label{lem app lemma 1}
    Let $M \in \calM (A((x)))$ be finitely generated. Let $m \in M$ and suppose that for every cofinite $I \in \calI_A$ the image of $m$ in $M / I M$ is equal to $0$. Then $m$ is equal to $0$.
\end{lemma}

\begin{proof}
    Let $J \subset A((x))$ be the ideal consisting of $g$ for which $gm = 0$. Denote by $J_A \subset A$ the subset consisting of lowest coefficients of elements in $J$ (say that the lowest coefficient of $0$ is $0$). Then $J_A$ is an ideal in $A$. We will show now that $J_A = A$; then $J$ contains an element with lowest coefficient $1$, so an invertible element, and therefore $m = 0$.
    
    \medskip
    
    Let $\frakm \subset A$ be a maximal ideal. By Krull's intersection theorem applied to $A((x))$ acting on $M$ and the ideal $\frakm \cdot A((x)) \subset A((x))$ (recall that $A((x))$ is Noetherian as $A$ is), and by the given, we have that there exists $f \in \frakm \cdot A((x))$ such that $(1+f)m = 0$, i.e. $1+f \in J$. So $J_A \not\subset \frakm$. Thus, $J_A$ is not contained in any maximal ideal, and hence $J_A = A$.
\end{proof}

\begin{lemma}\label{lem app lemma 2}
    Let $M \in \calM (A[x])$ be finitely generated. Let $$ m \in M^{\wedge}_x$$ and suppose that for every cofinite $I \in \calI_A$ the image of $m$ in $M^{\wedge}_x / I M^{\wedge}_x$ is $x$-torsion. Then $m$ is $x$-torsion.
\end{lemma}

\begin{proof}
    We want to show that the image of $m$ under $$ M^{\wedge}_x \to (M^{\wedge}_x)_x$$ is equal to $0$. For cofinite $I \in \calI_A$, we have a commutative diagram $$ \xymatrix{ M^{\wedge}_x \ar[r] \ar[d] & (M^{\wedge}_x)_x \ar[d] \\ M^{\wedge}_x / I M^{\wedge}_x\ar[r] & (M^{\wedge}_x)_x / I (M^{\wedge}_x)_x }.$$ Since $$(M^{\wedge}_x)_x / I (M^{\wedge}_x)_x \cong (M^{\wedge}_x / I M^{\wedge}_x)_x,$$ the image of $m$ when going down and then right is zero by the given, and thus so is the image when going right and then down. Notice now that $(M^{\wedge}_x)_x$ is a finitely-generated $A((x))$-module. Therefore, applying Lemma \ref{lem app lemma 1} to the image of $m$ in $(M^{\wedge}_x)_x$, we see that it is equal to $0$.
\end{proof}

\begin{lemma}\label{lem app lemma 3}
    Let $M \in \calM (A)$ be finitely generated and let $f \in A$. Let $$ m \in M^{\wedge}_f$$ and suppose that for every cofinite $I \in \calI_A$ the image of $m$ in $M^{\wedge}_f / I M^{\wedge}_f$ is $f$-torsion. Then $m$ is $f$-torsion.
\end{lemma}

\begin{proof}
    Let us consider $M$ as an $A[x]$-module, where the free variable $x$ acts as $f \in A$ does. Then $M^{\wedge}_x = M^{\wedge}_f$ and so on - the problem translates to that of Lemma \ref{lem app lemma 2}.
\end{proof}

\subsubsection{} The second auxiliary lemma is Lemma \ref{lem app lemma 4}.

\begin{lemma}\label{lem app lemma 4}
    Let $M \in \calM (A)$ be finitely generated, and let $f \in A$. The map $$ M^f \to (M^{\wedge}_f)^f$$ induced by the canonical $M \to M^{\wedge}_f$ is an isomorphism.
\end{lemma}

\begin{proof}
    We consider the exact sequence $$ 0 \to M^f \to M \to C \to 0$$ where $C$ stands for the cokernel. Applying completion, we obtain an exact sequence $$ 0 \to (M^f)^{\wedge}_f \to M^{\wedge}_f \to C^{\wedge}_f \to 0.$$ Applying torsion, we obtain an exact sequence $$ 0 \to ((M^f)^{\wedge}_f)^f \to (M^{\wedge}_f)^f \to (C^{\wedge}_f)^f.$$ Notice that $M^f \cong (M^f)^{\wedge}_f$ so that $M^f \cong ((M^f)^{\wedge}_f)^f$. Therefore the Lemma will be established if we show that $(C^{\wedge}_f)^f = 0$. To that end, let $c \in (C^{\wedge}_f)^f$, so $f^n c = 0$ for some $n \in \bbZ_{\ge 1}$. For every $i \in \bbZ_{\ge 1}$, let $c_i \in C$ be such that the images of $c$ and $c_i$ in $C / f^i C$ coincide. Then $f^n c_i \in f^i C$, and since $C$ has no $f$-torsion, for $i > n$ we get $c_i \in f^{i-n} C$. Therefore for $i>n$ we obtain $c_{i-n} \in c_i + f^{i-n} C \in f^{i-n} C$. In other words, the image of $c$ in $C / f^j C$ is equal to $0$ for all $j \in \bbZ_{\ge 1}$, i.e. $c = 0$.
\end{proof}

\subsection{Proof of Proposition \ref{prop app main 2}}

\subsubsection{}

Proposition \ref{prop app main 2} will clearly follow from Proposition \ref{prop app main} and the following lemma:

\begin{lemma}
	Assume that $Z$ is normal and that the dimension of $Z$ is strictly larger than $1$. Let $s \in S(\calO_Z)$ be generically coherent. Then $s$ is coherent.
\end{lemma}

\begin{proof}
	Let $0 \neq f \in \calO_Z$ and $b\in \calO_{Z_f}$ be such that, for every $0$-dimensional closed subscheme $W \subset Z_f$, we have $b|_{W} = s_W$.
	
	\medskip
	
	If there exists $a \in \calO_Z$ such that $a|_{Z_f} = b$, then consider $s^{\prime} := s - [a]$ (where $[a]$ is the image of $a$ under $\calO_Z \to S(\calO_Z)$). It is clear that, for every $W \in \calI_Z$, $(s^{\prime}_W)|_{W \cap Z_f} = 0$. Hence, by Lemma \ref{lem torsion coherent}, $s^{\prime}$ is coherent, and therefore so is $s$.
	
	\medskip
	
	It is therefore left to see that $b$ can be extended from $Z_f$ to $Z$. We write $b = \frac{a}{f^n}$ for some $a \in \calO_Z$ and $n \in \bbZ_{\ge 1}$. It is enough to show that $f$ divides $a$ in $\calO_Z$, since then we can decrease $n$ and, repeating, eventually get to $n = 0$.
	
	\medskip
	
	Let $U \subset Z$ be affine such that $U \cap \{ f = 0\}$ is smooth (here $\{ f = 0\}$ is the scheme-theoretical vanishing locus of $f$). It is enough to see that $a|_U$ is divisible by $f|_U$ since then the rational function $\frac{a}{f}$ on $Z$ is regular on an open subset whose complement is of codimension at least $2$, and therefore, as is well known, $\frac{a}{f}$ is in fact regular on the whole of $Z$, showing that $f$ divides $a$ in $\calO_Z$.
	
	\medskip
	
	Let $z$ be a closed point of $U \cap \{ f = 0\}$. Since the dimension of $U$ is bigger than $1$, we can find a closed subscheme $C \subset U$ which is irreducible of dimension $1$ passing through $z$ and with $C \cap Z_f \neq \emptyset$. Then, since $a|_C$ and $f^n s_C$ are equal when restricted to all $0$-dimensional closed subschemes of $C \cap Z_f$, and so equal on $C \cap Z_f$ which is dense in $C$, they must be equal - $a|_C = f^n s_C$. In particular, $a$ vanishes at $z$. We conclude that $a$ vanishes at all closed points of $U \cap \{ f = 0\}$. Since $U \cap \{ f = 0\}$ is smooth, this implies that $a|_U$ is divisible by $f|_U$.
\end{proof}
%%%%%%%%%%%%%%%%%%%%%%%%%%%%%%%%%%%%%%%%%%
%%%%%%%%%%%%%%%%%%%%%%%%%%%%%%%%%%%%%%%%%%

\end{document}